\documentclass[11pt,a4paper,reqno]{amsart} 
\usepackage{amssymb}
\usepackage{latexsym}
\usepackage{amsmath}
\usepackage{mathrsfs}
\usepackage{euscript}
\usepackage[X2,T1]{fontenc}
\usepackage{cite}
\usepackage{calc}                   

\newcommand{\scal}[2]{\langle #1,#2\rangle}
\newcommand{\rr}[1]{\mathbf R^{#1}}

\newcommand{\zz}[1]{\mathbf Z^{#1}}

\newcommand{\nn}[1]{\mathbf N^{#1}}
\newcommand{\nm}[2]{\Vert #1\Vert _{#2}}
\newcommand{\Nm}[2]{\left \Vert #1\right \Vert _{#2}}

\newcommand{\sets}[2]{\{ \, #1\, ;\, #2\, \} }

\newcommand{\fy}{\varphi}
\newcommand{\cdo}{\, \cdot \, }

\newcommand{\eabs}[1]{\langle #1\rangle}

\newcommand{\vrum}{\vspace{0.1cm}}

\newcommand{\maclD}{\mathcal D}
\newcommand{\maclE}{\mathcal E}

\newcommand{\maclS}{\mathcal S}
\newcommand{\mascB}{\mathscr B}
\newcommand{\mascD}{\mathscr D}

\newcommand{\mascF}{\mathscr F}

\newcommand{\mascP}{\mathscr P}
\newcommand{\mascS}{\mathscr S}

\newcommand{\mabfp}{{\boldsymbol p}}
\newcommand{\mabfq}{\boldsymbol q}
\newcommand{\mabfr}{\boldsymbol r}
\newcommand{\mabfz}{\boldsymbol z}

\renewcommand{\projlim}[1]{\underset{#1}{\operatorname{proj \, lim\, }}}

\newcommand{\sfW}{\mathsf{W}}

\numberwithin{equation}{section}          

\newtheorem{thm}{Theorem}
\numberwithin{thm}{section}

\newcommand{\rubrik}{}
\newtheorem{prop}[thm]{Proposition}
\newtheorem{cor}[thm]{Corollary}
\newtheorem{lemma}[thm]{Lemma}

\theoremstyle{definition}

\newtheorem{defn}[thm]{Definition}

\theoremstyle{remark}
\newtheorem{rem}[thm]{Remark}              

\title{The Zak transform on Gelfand-Shilov
and modulation spaces with applications to operator theory}

\author{Joachim Toft}

\address{Department of Mathematics,
Linn{\ae}us University, V{\"a}xj{\"o}, Sweden}

\email{joachim.toft@lnu.se}



\begin{document}

\par

\begin{abstract}
We characterize Gelfand-Shilov spaces, their
distribution spaces and modulation spaces in terms of
estimates of their Zak transforms.
We use these result for general investigations of quasi-periodic
functions and distributions. We also establish necessary and
sufficient conditions for linear operators in order for these
operators should be conjugations by the Zak transform.
\end{abstract}

\maketitle

\section{Introduction}\label{sec0}

\par

In the paper we characterise
Gelfand-Shilov spaces of functions and distributions,
modulation spaces and Gevrey classes in background of
mapping properties of the Zak transforms. We apply these
results to deduce duality properties of spaces of
quasi-periodic functions and distributions and for investigating
transitions of linear operators under the Zak transform.

\par

The Zak transforms are unpredictable and exciting
in several ways. They 
appear in natural ways when dealing with Gabor frame 
operators in
the cases of "critical sampling", where the Gabor theory
cease to work properly. This ought to be the reason why the
transform possess several exciting and almost magical
properties, useful in Gabor theory.

\par

For example, in critical sampling cases, the Zak
transform $Z$, adapted to the sampling parameters,
takes the Gabor frame operator $S_{\phi ,\psi}$
into the multiplication operator
$$
F\mapsto c\cdot \overline {Z\phi}\cdot Z\psi \cdot F
$$
for some constant $c$ which depends on the sampling
parameters. (See \cite{Ho1,ToNa} and Section \ref{sec1} for
notations.) We remark that this property is heavily
used when showing that Gabor atoms and their
canonical dual atoms often belong to the same function
classes. (See \cite{Bol,BoJa,Gc2}.)

\par

An other example concerns the fact that if $Zf$
is continuous, then it has zeros. This property
is important when deducing various kinds of
Balian-Low theorems, which are essential when
finding limitations for bases and Gabor frames
in Gabor analysis (see Theorem 8.4.1 and its
consequences in\cite{Gc2}).

\par

Before entering the Gabor theory, Zak transforms were
first introduced and used in a problem in differential
equation by Gelfand in \cite{Gel}. Subsequently, the transforms
were applied in various contexts, especially 
in solid state physics by Zak in \cite{Za1} and in
differential equations by Brezin in \cite{Bre}.

\par

In these considerations it is essential to understand
various kinds of mapping properties of the Zak transforms.
The transforms take suitable functions, defined on the
configuration space $\rr d$ into quasi-periodic functions,
defined on the phase space $\rr {2d}$. Hence, in similar ways as
for periodic functions, Zak transformed functions are
completely described by their behaviour on suitable
rectangles.

\par

For example, the (standard) Zak transform is given by
$$
(Z_1 f)(x,\xi ) \equiv \sum _{j\in \zz d}
f(x-j)e^{i\scal j\xi},
$$
when $f$ is a suitable function or distribution (see \eqref{Eq:DefZak} for the
general definition of the Zak transform). By the definition
it follows that if $F=Z_1f$ and $Q_{d,r}$ is the cube $[0,r]^d$,
then $F$ is quasi-periodic (with respect
to $Q_{d,1}\times Q_{d,2\pi}$). That is,
$$
F(x+k,\xi ) = e^{i\scal k\xi}F(x,\xi )
\quad \text{and}\quad
F(x,\xi +2\pi \kappa) = F(x,\xi ),\quad k,\kappa \in \zz d.
$$
It follows from these equalities that $F$
is completely reconstructable
from its data on $Q_{d,1}\times Q_{d,2\pi}$.

\par

It is well-known that $Z_1$ is bijective from $L^2(\rr d)$ to
the set of quasi-periodic elements in $L^2(Q_{d,1}
\times Q_{d,2\pi})$ and that
\begin{equation}\label{Eq:ZakonL2}
\nm {Z_1 f}{L^2(Q_{d,1}\times Q_{d,2\pi})} =
(2\pi) ^{\frac d2}\nm f{L^2},
\qquad f\in L^2(\rr d).
\end{equation}
(Cf. e.{\,}g. \cite[Theorem 8.2.3]{Gc2}.) Consequently,
$L^2(\rr d)$ can be characterized in a convenient
way by its image under the Zak transform.

\par

An other space that can be characterized by related
mapping properties concerns the Schwartz space
$\mascS (\rr d)$. In fact, it is proved in
\cite{Jan} by Janssen that $Z_1$ is continuous and
bijective from $\mascS (\rr d)$ to the set of
quasi-periodic elements in $C^\infty (\rr {2d})$.

\par

In \cite{Ti,TiHe}, Heil and Tinaztepe deduce
some important mapping properties for the Zak transform
on modulation spaces, and apply these
results to deduce Balian-Low properties in the framework
of such spaces.
These mapping
properties on modulation spaces seems not to be (complete)
characterizations, because of absence of bijectivity.
In fact, apart from the spaces
$L^2(\rr d)$ and $\mascS (\rr d)$,
the whole theory seems to lack characterizations of
essential function and distribution spaces
via the Zak transform as remarked in
Subsection 8.2 (f) in \cite{Gc2}.

\par

In Section \ref{sec2} we make this part more complete
and furnish the theory with various kinds of
characterizations. Especially we characterize
modulation and Lebesgue spaces by suitable Lebesgue
estimates of short-time Fourier transforms of the Zak 
transforms of the involved functions. 
We also
characterize the dual $\mascS '(\rr d)$ of
$\mascS (\rr d)$, the (standard) Gelfand-Shilov
spaces and their distribution spaces by their
images under the Zak transform.

\par

For example
we prove that $Z_1$ is continuous and bijective
from $\mascS '(\rr d)$ to the set of all
quasi-periodic distributions on
$Q_{d,1}\times Q_{d,2\pi}$. (See Theorem
\ref{Thm:ZakClassical}.) In Theorems \ref{Thm:ZakTestFunctions}
and \ref{Thm:ZakDist} we deduce similar characterizations for
Gelfand-Shilov spaces and their distribution
spaces. As a consequence of
Theorem \ref{Thm:ZakTestFunctions} we have
that the Zak transform $Z_1$ maps
the Gelfand-Shilov space $\maclS _s^\sigma (\rr d)$
bijectively to
$\maclE _{Z,1}^{\sigma ,s}(\rr {2d})$,
the set of all quasi-periodic functions (quasi-periodic distributions) on
$Q_{d,1}\times Q_{d,2\pi}$
in the Gevrey class $\maclE ^{\sigma ,s}(\rr {2d})$.
In the same way it follows from
Theorem \ref{Thm:ZakDist}
that the Zak transform $Z_1$ maps
the Gelfand-Shilov distribution space $(\maclS _s^\sigma )'(\rr d)$
bijectively to
$(\maclE _{Z,1}^{\sigma ,s})'(\rr {2d})$,
the set of all quasi-periodic distributions on
$Q_{d,1}\times Q_{d,2\pi}$ in $(\maclS _{s,\sigma}^{\sigma ,s})'(\rr {2d})$.
As a consequence,
if $s+\sigma <1$, then there are no non-trivial quasi-periodic functions
in $\maclE ^{\sigma ,s}(\rr {2d})$ (cf. Corollary \ref{Cor:ZakTestFunctions}).

\par

An other consequence of our 
results is that $Z_1$ maps the modulation space $M^p(\rr 
d)$ continuously and bijectively to the set of all
elements in $W^{\infty ,p}(\rr {2d})$ which are 
quasi-periodic on $Q_{d,1}\times Q_{d,2\pi}$.
Furthermore,
\begin{equation}\label{Eq:ZakModIntrRel}
\nm {Z_1f}{W^{\infty ,p}}\asymp \nm f{M^p}
\end{equation}
(see Theorem \ref{Thm:ZakModulation}
and Corollary \ref{Cor:ZakModulation}).

\par

We also use some recent results in
\cite{Toft26} on Wiener estimates to deduce
different versions of the latter characterization. For example we show that
\eqref{Eq:ZakModIntrRel} in combination with results in \cite[Section 2]{Toft26}
give
\begin{equation}\label{Eq:EquivZakMod}
f\in M^p(\rr d)\quad \Leftrightarrow \quad
V_\Phi (Z_1 f) \in L^p(Q_{d,1}\times Q_{d,2\pi}\times \rr d\times \rr d).
\end{equation}
If $p=2$, then an application of
Parseval's formula implies that \eqref{Eq:EquivZakMod} is the same as
$$
f\in M^2(\rr d)\quad \Leftrightarrow \quad
Z_1 f \in L^2(Q_{d,1}\times Q_{d,2\pi}),
$$
which is a slightly weaker form of \eqref{Eq:ZakonL2}.

\par

In Section \ref{sec3} we apply the mapping results of the Zak transform to
deduce duality properties for spaces of quasi-periodic functions and
distributions. For example, if $p\in [1,\infty )$ and $\frac 1p+\frac 1{p'}=1$, then
we prove that the dual of quasi-periodic elements in $\maclE ^{\sigma ,s}$ and
in $W^{\infty ,p}$ can be identified with the set of quasi-periodic elements in the
Gelfand-Shilov distribution space $(\maclS _{s,\sigma}^{\sigma ,s})'$,
respective in $W^{\infty ,p'}$. An essential part of these investigations
concern characterizations of quasi-periodic elements in terms of estimates
of their short-time Fourier transforms, given in the end of Section \ref{sec2} and
the beginning of Section \ref{sec3}.

\par

Finally, in Section \ref{sec4} we show how linear operators, $T$ are transformed
under conjugation of the Zak transform, $T_Z=Z_1\circ T \circ Z_1^{-1}$.
It follows from our investigations
 that the map $T\mapsto T_Z$ is a bijection between
the set of all continuous linear mappings
$$
T\, :\, \maclS _s^\sigma \to (\maclS _s^\sigma )'
$$
the set of all continuous linear mappings
$$
T_Z\, :\, \maclE ^{\sigma ,s}_{Z,1} \to (\maclE ^{\sigma ,s}_{Z,1})'
$$
(cf. Theorems \ref{Thm:ZakTransOpChar} and \ref{Thm:ZakTransOpChar2}).
At the same time we prove that a map $T_Z$ maps quasi-periodic functions or distributions
into quasi-periodic functions or distributions, if and only if $T_Z$ commutes
with each operator $U_{y,\eta}$, $y,\eta \in \rr d$, where
$$
(U_{y,\eta }F)(x,\xi )= e^{-i\scal y{\xi +\eta}}F(x+y,\xi +\eta ),
\qquad y,\eta \in \rr d.
$$

\par

\section*{Acknowledgement}

\par

I am very grateful to Professor Hans Feichtinger for reading parts of the paper
and giving valuable comments, leading to improvements of the content
and the style.

\par

\section{Preliminaries}\label{sec1}

\par

In this section we recall some basic facts. We start by discussing
Gelfand-Shilov spaces and their properties. Thereafter we recall
some properties of modulation spaces and discuss different aspects
of periodic distributions

\par

\subsection{Gelfand-Shilov spaces and Gevrey classes}\label{subsec1.1}
Let $0<s,\sigma \in \mathbf R$ be fixed. Then the Gelfand-Shilov
space $\mathcal S_{s}^\sigma (\rr d)$
($\Sigma _{s}^\sigma (\rr d)$) of Roumieu type (Beurling type) with parameters $s$
and $\sigma$ consists of all $f\in C^\infty (\rr d)$ such that
\begin{equation}\label{gfseminorm}
\nm f{\mathcal S_{s,h}^\sigma }\equiv \sup \frac {|x^\alpha \partial ^\beta
f(x)|}{h^{|\alpha  + \beta |}\alpha !^s \, \beta !^\sigma}
\end{equation}
is finite for some $h>0$ (for every $h>0$). Here the supremum should be taken
over all $\alpha ,\beta \in \mathbf N^d$ and $x\in \rr d$. We equip
$\mathcal S_{s}^\sigma (\rr d)$ ($\Sigma _{s}^\sigma (\rr d)$) by the canonical inductive limit
topology (projective limit topology) with respect to $h>0$, induced by
the semi-norms in \eqref{gfseminorm}.

\par

\medspace

The \emph{Gelfand-Shilov distribution spaces} $(\mathcal S_s^{\sigma})'(\rr d)$
and $(\Sigma _s^\sigma)'(\rr d)$ are the dual spaces of $\mathcal S_s^{\sigma}(\rr d)$
and $\Sigma _s^\sigma (\rr d)$, respectively.  As for the Gelfand-Shilov spaces there 
is a canonical projective limit topology (inductive limit topology) for $(\maclS _s^{\sigma})'(\rr d)$ 
($(\Sigma _s^\sigma)'(\rr d)$).(Cf. \cite{GS, Pil1, Pil3}.)
For conveniency we set
$$
\maclS _s=\maclS _s^s,\quad \maclS _s'=(\maclS _s^s)',\quad
\Sigma _s=\Sigma _s^s
\quad \text{and}\quad
\Sigma _s'=(\Sigma _s^s)'.
$$

\par

From now on we let $\mathscr F$ be the Fourier transform which
takes the form
$$
(\mathscr Ff)(\xi )= \widehat f(\xi ) \equiv (2\pi )^{-\frac d2}\int _{\rr
{d}} f(x)e^{-i\scal  x\xi }\, dx
$$
when $f\in L^1(\rr d)$. Here $\scal \cdo \cdo$ denotes the usual
scalar product on $\rr d$. The map $\mathscr F$ extends 
uniquely to homeomorphisms on $\mathscr S'(\rr d)$,
from $(\mathcal S_s^\sigma )'(\rr d)$ to $(\mathcal S_\sigma ^s)'(\rr d)$ and
from $(\Sigma _s^\sigma)'(\rr d)$ to $(\Sigma _\sigma ^s)'(\rr d)$. Furthermore,
$\mascF$ restricts to
homeomorphisms on $\mathscr S(\rr d)$, from
$\mathcal S_s^\sigma (\rr d)$ to $\mathcal S_\sigma ^s(\rr d)$ and
from $\Sigma _s^\sigma (\rr d)$ to $\Sigma _\sigma ^s(\rr d)$,
and to a unitary operator on $L^2(\rr d)$. 

\medspace

Next we consider a more general class of Gelfand-Shilov spaces and their distribution
spaces. Let $0\le s_1,s_2,\sigma _1,\sigma _2\in \mathbf R$ be fixed. Then the Gelfand-Shilov
space $\mathcal S_{s_1,s_2}^{\sigma _1,\sigma _2}(\rr {d_1+d_2})$
($\Sigma _{s_1,s_2}^{\sigma _1,\sigma _2}(\rr {d_1+d_2})$) of Roumieu type (Beurling type)
with parameters $s_1$, $s_2$, $\sigma _1$ and $\sigma _2$ consists of all $f\in C^\infty
(\rr {d_1+d_2})$ such that
\begin{equation}\label{gfseminorm2}
\nm f{\mathcal S _{s_1,s_2,h}^{\sigma _1,\sigma _2}}
\equiv \sup \frac {|x_1^{\alpha _1}x_2^{\alpha _2}
\partial _{x_1}^{\beta _1}\partial _{x_2}^{\beta _2}
f(x_1,x_2)|}{h^{|\alpha _1+\alpha _2 + \beta _1+\beta _2|}\alpha _1!^{s_1}\,
\alpha _2!^{s_2}\, \beta _1!^{\sigma _1}\, \beta _2!^{\sigma _2}}
\end{equation}
is finite for some $h>0$ (for every $h>0$). Here the supremum should be taken
over all $\alpha _j,\beta _j\in \mathbf N^{d_j}$ and $x_j\in \rr {d_j}$, $j=1,2$. We
equip $\mathcal S_{s_1,s_2}^{\sigma _1,\sigma _2}(\rr {d_1+d_2})$
($\Sigma _{s_1,s_2}^{\sigma _1,\sigma _2}(\rr {d_1+d_2})$) by the canonical inductive limit
topology (projective limit topology) with respect to $h>0$, induced by
the semi-norms in \eqref{gfseminorm2}.

\par

The space $\Sigma _{s_1,s_2}^{\sigma _1,\sigma _2}(\rr {d_1+d_2})$ is a Fr{\'e}chet space
when the topology is induced by the seminorms
$\nm \cdo{\mathcal S_{s_1,s_2,h}^{\sigma _1,\sigma _2}}$, $h>0$.

\medspace

The \emph{Gelfand-Shilov distribution spaces} $(\maclS _{s_1,s_2}
^{\sigma _1,\sigma _2})'(\rr {d_1+d_2})$
and $(\Sigma _{s_1,s_2}^{\sigma _1,\sigma _2})'(\rr {d_1+d_2})$ are the dual spaces of
$\maclS _{s_1,s_2}^{\sigma _1,\sigma _2}(\rr {d_1+d_2})$
and $\Sigma _{s_1,s_2}^{\sigma _1,\sigma _2}(\rr {d_1+d_2})$, respectively. Evidently,
$\maclS _{s_1,s_2}^{\sigma _1,\sigma _2}(\rr {d_1+d_2})$,
$\Sigma _{s_1,s_2}^{\sigma _1,\sigma _2}(\rr {d_1+d_2})$
and their duals possess similar topological properties as
$\maclS _{s}^{\sigma }(\rr {d})$, $\Sigma _{s}^{\sigma }(\rr {d})$ and their duals.
The space $\Sigma _{s_{1},s_{2}}^{\sigma _{1},\sigma _{2}}(\rr {d_1+d_2})$
is a Fr{\'e}chet space
with seminorms $\nm \cdo{\maclS _{s_{1},s_{2},h}^{\sigma _{1},\sigma _{2}}}$, $h>0$. 
Moreover,
$\Sigma _{s_{1},s_{2}}^{\sigma _{1},\sigma _{2}}(\rr {d_1+d_2})\neq \{ 0\}$, if and only if 
$s_j+\sigma _j\ge 1$ and
$(s_j,\sigma _j)\neq (\frac 12,\frac 12)$, $j=1,2$, and
$\maclS _{s_{1},s_{2}}^{\sigma _{1},\sigma _{2}}
(\rr {d_1+d_2})\neq \{ 0\}$, if and only if $s_j+\sigma _j\ge 1$, $j=1,2$.
By $(j+k)!\le 2^{j+k}j!k!$ when $j,k\ge 0$ are integers we get
$\maclS _{s,s}^{\sigma ,\sigma } = \maclS _s^\sigma$ and $\Sigma _{s,s}^{\sigma ,\sigma}
= \Sigma _s^\sigma$.

\par

For any $s_j,\sigma _j,s_{0,j},\sigma _{0,j}>0$ such that $s_j>s_{0,j}$,
$\sigma _j>\sigma _{0,j}$ and $s_{0,j}+\sigma _{0,j}\ge 1$, $j=1,2$,
we have
\begin{multline}\label{GSembeddings}
\maclS _{s_{0,1},s_{0,2}}^{\sigma _{0,1},\sigma _{0,2}}(\rr {d_1+d_2})
\hookrightarrow
\Sigma _{s_{1},s_{2}}^{\sigma _{1},\sigma _{2}}(\rr {d_1+d_2})
\hookrightarrow 
\maclS _{s_{1},s_{2}}^{\sigma _{1},\sigma _{2}}(\rr {d_1+d_2})
\\[1ex]
\hookrightarrow 
\mascS (\rr {d_1+d_2})
\hookrightarrow 
\mascS '(\rr {d_1+d_2})
\hookrightarrow
(\maclS _{s_{1},s_{2}}^{\sigma _{1},\sigma _{2}})' (\rr {d_1+d_2})
\\[1ex]
\hookrightarrow
(\Sigma _{s_{1},s_{2}}^{\sigma _{1},\sigma _{2}})'(\rr {d_1+d_2})
\hookrightarrow
(\maclS _{s_{0,1},s_{0,2}}^{\sigma _{0,1},\sigma _{0,2}})'(\rr {d_1+d_2}),
\end{multline}
with dense embeddings.
Here and in what follows we use the notation $A\hookrightarrow B$ 
when the topological spaces $A$ and $B$ satisfy $A\subseteq B$ with
continuous embeddings.

\par

Let $\mascF _jf$ denote the partial Fourier transform of
$f(x_1,x_2)\in \mascS (\rr {d_1+d_2})$
with respect to $x_j$, $j=1,2$.
Then $\mascF _1$ and $\mascF _2$ extend uniquely to homeomorphisms
\begin{alignat*}{4}
\mascF _1 & \ :\  &   &(\maclS _{s_1,s_2}^{\sigma _1,\sigma _2})'(\rr {d_1+d_2})
&   &\ \to \ &
&(\maclS _{\sigma _1,s_2}^{s_1,\sigma _2})'(\rr {d_1+d_2})
\\[1ex]
\mascF _2 & \ :\  &   &(\maclS _{s_1,s_2}^{\sigma _1,\sigma _2})'(\rr {d_1+d_2})
&   &\ \to \ &
&(\maclS _{s_1,\sigma _2}^{\sigma _1,s_2})'(\rr {d_1+d_2}),
\intertext{respectively, and restricts to homeomorphisms}
\mascF _1 & \ :\  &   &\maclS _{s_1,s_2}^{\sigma _1,\sigma _2}(\rr {d_1+d_2})
&   &\ \to \ &
&\maclS _{\sigma _1,s_2}^{s_1,\sigma _2}(\rr {d_1+d_2})
\\[1ex]
\mascF _2 & \ :\  &   &\maclS _{s_1,s_2}^
{\sigma _1,\sigma _2}(\rr {d_1+d_2})
&   &\ \to \ &
&\maclS _{s_1,\sigma _2}^{\sigma _1,s_2}(\rr {d_1+d_2}),
\end{alignat*}
The same holds true after each Gelfand-Shilov function or 
distribution space
of Roumieu type have been replaced by corresponding Beurling 
type space.

\par

Gelfand-Shilov spaces can in convenient
ways be characterized in terms of estimates of the
involved functions and their Fourier
transforms. More precisely, in \cite{ChuChuKim, Eij}
it is proved that
if $f\in \mascS '(\rr d)$ and $s,\sigma >0$, then
$f\in \maclS _s^\sigma (\rr d)$
($f\in \Sigma _s^\sigma (\rr d)$), if and only if
\begin{equation}\label{Eq:GSFtransfChar}
|f(x)|\lesssim e^{-r|x|^{\frac 1s}}
\quad \text{and}\quad
|\widehat f(\xi )|\lesssim e^{-r|\xi |^{\frac 1\sigma }},
\end{equation}
for some $r>0$ (for every $r>0$).
Here $f(\theta )\lesssim g(\theta )$ means that
$f(\theta )\le cg(\theta)$ for some
constant $c>0$ which is independent of $\theta$ in the
domain of $f$ and $g$.
We also set $f(\theta )\asymp g(\theta )$ when
$f(\theta )\lesssim g(\theta )$
and $g(\theta )\lesssim f(\theta )$. More generally,
it follows from \cite{ChuChuKim}
that if $f\in \mascS '(\rr {d_1+d_2})$ and
$s_1,s_2,\sigma _1,\sigma _2>0$, then
$f\in \maclS _{s_1,s_2}^{\sigma _1,\sigma _2}(\rr {d_1+d_2})$
($f\in \Sigma _{s_1,s_2}^{\sigma _1,\sigma _2}
(\rr {d_1+d_2})$), if and only if
\begin{equation}\tag*{(\ref{Eq:GSFtransfChar})$'$}
|f(x_1,x_2)|
\lesssim
e^{-r(|x_1|^{\frac 1{s_1}} + |x_2|^{\frac 1{s_2}})}
\quad \text{and}\quad
|\widehat f(\xi _1,\xi _2)|
\lesssim
e^{-r(|\xi _1|^{\frac 1{\sigma _1}} + |\xi _2|^{\frac 1{\sigma _2}})},
\end{equation}
for some $r>0$ (for every $r>0$).

\par

Gelfand-Shilov spaces and their distribution spaces can also
be characterized by estimates of short-time Fourier
transforms, (see e.{\,}g. \cite{GZ,Toft18}).
More precisely, let $\phi \in \maclS _s (\rr d)$ be fixed. Then the \emph{short-time
Fourier transform} $V_\phi f$ of $f\in \maclS _s '
(\rr d)$ with respect to the \emph{window function} $\phi$ is
the Gelfand-Shilov distribution on $\rr {2d}$, defined by
$$
V_\phi f(x,\xi )  =
\mascF (f \, \overline {\phi (\cdo -x)})(\xi ).
$$
If $f ,\phi \in \maclS _s (\rr d)$, then it follows that
$$
V_\phi f(x,\xi ) = (2\pi )^{-\frac d2}\int f(y)\overline {\phi
(y-x)}e^{-i\scal y\xi}\, dy .
$$

\par

By \cite[Theorem 2.3]{Toft10} it follows that the map
$(f,\phi)\mapsto V_{\phi} f$ from $\mascS (\rr d) \times \mascS (\rr d)$ 
to $\mascS(\rr {2d})$ is uniquely extendable to a continuous map from 
$(\maclS_{s}^{\sigma})'(\rr d)\times (\maclS_{s}^{\sigma})'(\rr d)$
to $(\maclS_{s,\sigma}^{\sigma ,s})'(\rr {2d})$, and restricts to a continuous map
from $\maclS_{s}^{\sigma}(\rr d)\times \maclS_{s}^{\sigma}(\rr d)$
to $\maclS_{s,\sigma}^{\sigma ,s}(\rr {2d})$.

\par

The same conclusion holds with $\Sigma _s^\sigma$ and $\Sigma_{s,\sigma}^{\sigma ,s}$ 
in place of $\maclS _{s}^{\sigma}$ and $\maclS _{s,\sigma}^{\sigma ,s}$, respectively, at each place.

\par

The following results characterize Gelfand-Shilov spaces and their distribution
spaces in terms of estimates of short-time Fourier transform.

\par

\begin{prop}\label{stftGelfand2}
Let $f$ be a Gelfand-Shilov distribution on $\rr {d_1+d_2}$,
$$
s_j,\sigma _j>0\quad \text{and}\quad
\phi \in \maclS _{s_1,s_2}^{\sigma _1,\sigma _2}(\rr {d_1+d_2})\setminus 0
\quad
(\phi \in \Sigma _{s_1,s_2}^{\sigma _1,\sigma _2}(\rr {d_1+d_2})\setminus 0).
$$
Then the following is true:
\begin{enumerate}
\item
$f\in  \maclS _{s_1,s_2}^{\sigma _1,\sigma _2}(\rr {d_1+d_2})$
($f\in  \Sigma _{s_1,s_2}^{\sigma _1,\sigma _2}(\rr {d_1+d_2})$), if and only if
\begin{equation}\label{stftexpest2}
|V_\phi f(x_1,x_2,\xi _1,\xi _2)| \lesssim
e^{-r (|x_1|^{\frac 1{s_1}} + |x_2|^{\frac 1{s_2}}
+ |\xi _1|^{\frac 1{\sigma _1}} + |\xi _2|^{\frac 1{\sigma _2}})},
\end{equation}
for some $r > 0$ (for every $r>0$).

\vrum

\item
$f\in (\maclS _{s_1,s_2}^{\sigma _1,\sigma _2})'(\rr {d_1+d_2})$
($f\in (\Sigma _{s_1,s_2}^{\sigma _1,\sigma _2})'(\rr {d_1+d_2})$), if and only if
\begin{equation}\label{stftexpest2dist}
|V_\phi f(x_1,x_2,\xi _1,\xi _2)| \lesssim
e^{r (|x_1|^{\frac 1{s_1}} + |x_2|^{\frac 1{s_2}}
+ |\xi _1|^{\frac 1{\sigma _1}} + |\xi _2|^{\frac 1{\sigma _2}})},
\end{equation}
for every $r > 0$ (for some $r>0$).
\end{enumerate}
\end{prop}

\par

A proof of Proposition \ref{stftGelfand2} (1) can be found in
e.{\,}g. \cite{GZ} (cf. \cite[Theorem 2.7]{GZ}) and a proof of
Proposition \ref{stftGelfand2} (2) in the case $d_2=0$
can be found in \cite{Toft18}. The general case of Proposition \ref{stftGelfand2} (2)
follows by similar arguments as in \cite{Toft18} and is left for the reader.
See also \cite{CPRT10} for related results.

\par

Next we consider Gevrey classes on $\rr d$. Let $\sigma \ge 0$.
For any compact set $K\subseteq \rr d$, $h>0$ and $f\in C^{\infty}(K)$ let
\begin{equation}\label{e2}
\nm {f}{K,h,\sigma} \equiv \underset{\alpha\in \nn d}\sup
\left (
\frac{\nm {\partial^{\alpha}f}{L^{\infty}(K)}}{h^{\vert \alpha\vert}\alpha ! ^\sigma}
\right ) .
\end{equation} 
The Gevrey class $\maclE ^\sigma (K)$ ($\maclE ^{\sigma ;0}(K)$) of order $\sigma$ and
of Roumieu type (of Beurling type) is the set of all
$f\in C^{\infty}(K)$ such that \eqref{e2} is finite for some (for every)
$h>0$. We equipp $\maclE ^\sigma (K)$ ($\maclE ^{\sigma ;0}(K)$) by
the inductive (projective) limit topology with respect to $h>0$,
supplied by the seminorms in \eqref{e2}. Finally if $\lbrace
K_j\rbrace_{j\geq 1}$ is an exhausted
set of compact subsets of $\rr d$, then let
\begin{alignat*}{3}
\maclE ^\sigma (\rr d) &= \projlim{j} \maclE ^\sigma (K_j)
& \quad &\text{and} &\quad
\maclE ^{\sigma ;0}(\rr d) &= \projlim{j} \maclE ^{\sigma ;0}(K_j).
\intertext{In particular,}
\maclE ^\sigma (\rr d) &=\underset{j\geq 1}\bigcap \maclE ^\sigma (K_j)
& \quad &\text{and} &\quad
\maclE ^{\sigma ;0}(\rr d) &= \underset{j\geq 1}\bigcap \maclE ^{\sigma ;0}(K_j).
\end{alignat*}
It is clear that $\maclE ^{0;0}(\rr d)$ contains all constant functions
on $\rr d$, and that $\maclE ^{0}(\rr d)\setminus \maclE ^{0;0}(\rr d)$
contains all non-constant trigonometric polynomials.

\par

\subsection{Ordered, dual and phase split bases}\label{subsec1.2}

\par

Our discussions involving Zak transforms, periodicity,
modulation spaces and Wiener spaces are done in terms of suitable bases.

\par

\begin{defn}\label{Def:OrdBasis}
Let $E = \{ e_1,\dots,e_d \}$ be an \emph{ordered} basis of
$\rr {d}$. Then $E'$ denotes the basis of $e_1',\dots,e_d'$ 
in $\rr {d}$ which satisfies
$$
\scal {e_j} {e'_k}=2\pi \delta_{jk}
\quad \text{for every}\quad
j,k =1,\dots, d.
$$
The corresponding lattices are given by
\begin{align*}
\Lambda _E
&=
\sets{n_1e_1+\cdots+n_de_d}{(n_1,\dots,n_d)\in \zz d},
\intertext{and}
\Lambda'_E
&=
\Lambda_{E'}=\sets{\nu _1e'_1+\cdots+\nu _de'_d}{(\nu _1,\dots,\nu _d)
\in \zz d}.
\end{align*}
The sets $E'$ and $\Lambda '_E$ are
called the dual basis and dual lattice of $E$ and $\Lambda _E$, respectively.
If $E_1,E_2$ are ordered bases of $\rr d$ such that a permutation of
$E_2$ is the dual basis for $E_1$, then the pair $(E_1,E_2)$ are called
\emph{permuted dual bases} (to each others on $\rr d$).
\end{defn}

\par

\begin{rem}\label{Rem:TEMap}
Evidently, if $E$ is the same as in Definition \ref{Def:OrdBasis},
then there is a matrix $T_E$ with $E$ as the image of
the standard basis in $\rr d$. Then $E'$ is the image of the standard basis
under the map $T_{E'}= 2\pi(T^{-1}_E)^t$.
\end{rem}
%
%

\par

\begin{defn}\label{Def:FromTwoToOneBasis}
Let $E_1,E_2$ be ordered bases of $\rr d$,
$$
V_1=\sets {(x,0)\in \rr {2d}}{x\in \rr d},
\quad
V_2=\sets {(0,\xi )\in \rr {2d}}{\xi \in \rr d}
$$
and let $\pi _j$ from $\rr {2d}$ to $\rr d$, $j=1,2$, be the projections
$$
\pi _1(x,\xi ) = x
\quad \text{and}\quad
\pi _2(x,\xi ) = \xi .
$$
Then $E_1\times E_2$ is the ordered basis $\{ e_1,\dots ,e_{2d}\}$
of $\rr {2d}$ such that
\begin{alignat*}{2}
\{ e_1,\dots ,e_d\} &\subseteq V_1,
&\qquad
E_1 &= \{ \pi _1(e_1),\dots ,\pi _1(e_d)\} ,
\\[1ex]
\{ e_{d+1},\dots ,e_{2d}\}
&\subseteq V_2 &
\quad \text{and}\quad
E_2 &= \{ \pi _2(e_{d+1}),\dots ,\pi _2(e_{2d})\} .
\end{alignat*}
\end{defn}

\par

In the phase space it is convenient to consider phase split bases, which
are defined as follows.

\par

\begin{defn}\label{Def:Phasesplit}
Let $V_1$, $V_2$, $\pi _1$ and $\pi _2$ be as in Definition
\ref{Def:FromTwoToOneBasis}, $E$ be an ordered basis of
the phase space $\rr {2d}$ and let
$E_0\subseteq E$. Then
$E$ is called \emph{phase split} (with respect to $E_0$),
if the following is true:
\begin{enumerate}
\item the span of $E_0$ and $E\setminus E_0$ equal
$V_1$ and $V_2$, respectively;

\vrum

\item if $E_1=\pi _1(E_0)$ and $E_2=
\pi _2(E\setminus E_0)$, then $(E_1,E_2)$
are permuted dual bases.
\end{enumerate}
If $E$ is a phase split basis with respect to $E_0$ and that $E_0$ consists of
the first $d$ vectors in $E$, then $E$ is called \emph{strongly phase split}
(with respect to $E_0$).

\end{defn}

\par

In Definition \ref{Def:Phasesplit} it is understood that the orderings of
$E_0$ and $E\setminus E_0$
are inherited from the ordering in $E$.

\par

\begin{rem}\label{Rem:CommentPhasSplit}
Let $E$ and $E_j$, $j=0,1,2$ be the same as in Definition \ref{Def:Phasesplit}.
It is evident that $E_0$ and $E\setminus E_0$  consist of $d$ elements, and
that $E_1$ and $E_2$ are uniquely defined if the orders of $E_0$ and
$E\setminus E_0$ are preserved. The pair $(E_1,E_2)$ is called
\emph{the pair of permuted dual bases, induced by $E$ and $E_0$}.

\par

On the other hand, suppose that $(E_1,E_2)$ is a pair of permuted dual bases to each
others on $\rr d$. Then it is clear that for
$E_1\times E_2=\{ e_1,\dots ,e_{2d}\}$ in Definition \ref{Def:FromTwoToOneBasis}
and $E_0=\{ e_1,\dots ,e_d\}$, we have that
$E_0$ and $E$ fullfils all properties in Definition \ref{Def:Phasesplit}. In this case,
$E_1\times E_2$ is called the phase split basis (of $\rr {2d}$) induced by $(E_1,E_2)$.

\par

It follows that if $E'$, $E_1'$ and $E_2'$ are the dual bases of $E$, $E_1$ and $E_2$,
repsectively, then $E'=E_1'\times E_2'$.
\end{rem}

\par

\subsection{Invariant quasi-Banach spaces and spaces of
mixed quasi-normed spaces of Lebesgue types}\label{subsec1.3}

\par

We recall that a quasi-norm $\nm {\cdo}{\mascB}$ of order $r \in (0,1]$ on the
vector-space $\mascB$ over $\mathbf C$ is a nonnegative functional on
$\mascB$ which satisfies
\begin{alignat}{2}
 \nm {f+g}{\mascB} &\le 2^{\frac 1r-1}(\nm {f}{\mascB} + \nm {g}{\mascB}), &
\quad f,g &\in \mascB ,
\label{Eq:WeakTriangle1}
\\[1ex]
\nm {\alpha \cdot f}{\mascB} &= |\alpha| \cdot \nm f{\mascB},
& \quad \alpha &\in \mathbf{C},
\quad  f \in \mascB
\notag
\intertext{and}
   \nm f {\mascB} &= 0\quad  \Leftrightarrow \quad f=0. & &
\notag
\end{alignat}
The space $\mascB$ is then called a quasi-normed space. A complete
quasi-normed space is called a quasi-Banach space. If $\mascB$
is a quasi-Banach space with
quasi-norm satisfying \eqref{Eq:WeakTriangle1}
then by \cite{Aik,Rol} there is an equivalent quasi-norm to $\nm \cdo {\mascB}$
which additionally satisfies
\begin{align}\label{Eq:WeakTriangle2}
\nm {f+g}{\mascB}^r \le \nm {f}{\mascB}^r + \nm {g}{\mascB}^r, 
\quad f,g \in \mascB .
\end{align}
From now on we always assume that the quasi-norm of the quasi-Banach space $\mascB$
is chosen in such way that both \eqref{Eq:WeakTriangle1} and \eqref{Eq:WeakTriangle2}
hold.

\par

Before giving the definition of $v$-invariant spaces, we recall some facts on weight
functions.

\par

A \emph{weight} or \emph{weight function} on $\rr d$ is a
positive function $\omega
\in  L^\infty _{loc}(\rr d)$ such that $1/\omega \in  L^\infty _{loc}(\rr d)$.
The weight $\omega$ is called \emph{moderate},
if there is a positive weight $v$ on $\rr d$ such that
\begin{equation}\label{moderate}
\omega (x+y) \lesssim \omega (x)v(y),\qquad x,y\in \rr d.
\end{equation}
If $\omega$ and $v$ are weights on $\rr d$ such that
\eqref{moderate} holds, then $\omega$ is also called
\emph{$v$-moderate}.
We note that \eqref{moderate} implies that $\omega$ fulfills
the estimates
\begin{equation}\label{moderateconseq}
v(-x)^{-1}\lesssim \omega (x)\lesssim v(x),\quad x\in \rr d.
\end{equation}
We let $\mascP _E(\rr d)$ be the set of all moderate weights on $\rr d$.

\par

By \cite{Gc2.5} it follows that if $\omega \in \mascP _E(\rr d)$, then
$$
e^{-r|x|}\lesssim \omega (x)\lesssim e^{r|x|}
\quad \text{and}\quad
\omega (x+y)\lesssim \omega (x)e^{r|y|},\quad x,y\in \rr d.
$$
for some $r>0$.

\par

We say that $v$ is
\emph{submultiplicative} if $v$ is even and \eqref{moderate}
holds with $\omega =v$. In the sequel, $v$ and $v_j$ for
$j\ge 0$, always stand for submultiplicative weights if
nothing else is stated.

\par

In the following we define a broad family of mixed quasi-normed Lebesgue spaces.
Here $Q_E$ denotes the closed parallelepiped (or the \emph{$E$-cube}) is
spanned by the ordered basis $E$ of $\rr d$.

\par

\begin{defn}\label{Def:DiscLebSpaces}
Let $E = \{  e_1,\dots ,e_d \}$ be an ordered basis of $\rr d$, $Q_E$ be the
parallelepiped spanned by $E$, $\omega \in \mascP _E(\rr d)$
$\mabfq =(q_1,\dots ,q_d)\in (0,\infty ]^{d}$ and $r=\min (1,\mabfq )$.
If  $f\in L^r_{loc}(\rr d)$, then
$$
\nm f{L^{\mabfq }_{E,(\omega )}}\equiv
\nm {g_{d-1}}{L^{q_{d}}(\mathbf R)}
$$
where  $g_k\, :\, \rr {d-k}\to \mathbf R$,
$k=0,\dots ,d-1$, are inductively defined as
\begin{align*}
g_0(x_1,\dots ,x_{d})
&\equiv
|f(x_1e_1+\cdots +x_{d}e_d)\omega (x_1e_1+\cdots +x_{d}e_d)|,
\\[1ex]
\intertext{and}
g_k(\mabfz _k) &\equiv
\nm {g_{k-1}(\cdo ,\mabfz_k)}{L^{q_k}(\mathbf R)},
\quad \mabfz _k\in \rr {d-k},\ k=1,\dots ,d-1.
\end{align*}
If $\Omega \subseteq \rr d$ is measurable,
then $L^{\mabfq }_{E,(\omega )}(\Omega )$ consists
of all $f\in L^r_{loc}(\Omega )$ with finite quasi-norm
$$
\nm f{L^{\mabfq}_{E,(\omega )}(\Omega )}
\equiv
\nm {f_\Omega }{L^{\mabfq}_{E,(\omega )}(\rr d)},
\qquad
f_\Omega (x)
\equiv
\begin{cases}
f(x), &\text{when}\ x\in \Omega
\\[1ex]
0, &\text{when}\ x\notin \Omega .
\end{cases}
$$
The space $L^{\mabfq }_{E,(\omega )}(\Omega )$ is called 
\emph{$E$-split Lebesgue space (with respect to $\omega$, $\mabfq$
and $\Omega$)}.
\end{defn}

\par

We let $\ell _0' (\Lambda ' _E)$ be the set of all formal sequences
$a= \{ a(j) \} _{j\in \Lambda _E}\subseteq \mathbf C$ on $\Lambda _E$
and
$\ell ^\mabfp _{E,(\omega )}(\Lambda _E)$ be the \emph{discrete
version} of $L^\mabfp _{E,(\omega )}(\rr d)$ when
$\mabfp \in (0,\infty ]^d$. That is, $\ell ^\mabfp _{E,(\omega )}(\Lambda _E)$
consists of all formal sequences $a\in \ell _0' (\Lambda ' _E)$ such that
$$
\nm a{\ell ^\mabfp _{E,(\omega )}}
=
\nm a{\ell ^\mabfp _{E,(\omega )}(\Lambda _E)}
\equiv
\Nm {\sum _{j\in \Lambda _E}a(j)\chi _{j+Q_E}}{L^{\mabfq}_{E,(\omega )}(\rr d)}
$$
is finite. Here $\chi _\Omega$ is the characteristic function of the set $\Omega$.


\par

\begin{rem}\label{Rem:LebExpIdent}
Evidently, $L^{\mabfq}_{E,(\omega )} (\Omega )$ and
$\ell ^{\mabfq}_{E,(\omega )} (\Lambda )$
in Definition \ref{Def:DiscLebSpaces} are quasi-Banach spaces of order
$\min (\mabfp,1)$. We set
$$
L^{\mabfq}_{E} = L^{\mabfq}_{E,(\omega )}
\quad \text{and}\quad
\ell ^{\mabfq}_{E} = \ell ^{\mabfq}_{E,(\omega )}
$$
when $\omega =1$. For conveniency we identify
$\mabfq = (q,\dots ,q)\in (0,\infty ]^d$
with $q\in (0,\infty ]$ when considering spaces involving
Lebesgue exponents. In particular,
\begin{alignat*}{5}
L^{q}_{E,(\omega )} &=  L^{\mabfq}_{E,(\omega )},
&\quad
L^{q}_{E} &= L^{\mabfq}_{E},
&\quad
\ell ^{q}_{E,(\omega )} &= \ell ^{\mabfq}_{E,(\omega )}
&\quad &\text{and} &\quad 
\ell ^{q}_{E} &= \ell ^{\mabfq}_{E}
\intertext{for such $\mabfq$, and notice that these spaces agree with}
&{\phantom =}L^q_{(\omega )}, &\qquad
&{\phantom =}L^q, &\qquad
&{\phantom =}\ell ^{q}_{(\omega )}
&\quad &\text{and}&\quad
&{\phantom =}\ell ^{q},
\end{alignat*}
respectively, with equivalent quasi-norms.
\end{rem}

\par

\subsection{Modulation and Wiener spaces}\label{subsec1.4}

\par

We consider the following broad family of modulation spaces which contains
the classical modulation spaces, introduced by Feichtinger in \cite{F1}.

\par

\begin{defn}\label{Def:ExtClassModSpaces}
Let $\mabfp ,\mabfq \in (0,\infty ]^d$, $E_1$ and $E_2$ be ordered bases
of $\rr d$, $E=E_1\times E_2$,
$\phi \in \Sigma _1(\rr d)\setminus 0$ and let
$\omega \in \mascP _E(\rr {2d})$.
For any $f\in \Sigma _1'(\rr d)$ set
\begin{multline*}
\nm f{M^{\mabfp ,\mabfq}_{E,(\omega )}}
\equiv
\nm {H_{1,f,E_1,\mabfp ,\omega }}{L^{\mabfq}_{E_2}},
\\[1ex]
\quad \text{where}\quad
H_{1, f,E_1,\mabfp ,\omega }(\xi )
\equiv
\nm {V_\phi f (\cdo ,\xi )\omega (\cdo ,\xi )}{L^{\mabfp}_{E_1}}
\end{multline*}
and
\begin{multline*}
\nm f{W^{\mabfp ,\mabfq}_{E,(\omega )}}
\equiv
\nm {H_{2,f,E_2,\mabfq ,\omega }}{L^{\mabfp}_{E_1}},
\\[1ex]
\quad \text{where}\quad
H_{2, f,E_2,\mabfq ,\omega }(x)
\equiv
\nm {V_\phi f (x,\cdo )\omega (x,\cdo )}{L^{\mabfq}_{E_2}}.
\end{multline*}
The \emph{modulation space}
$M^{\mabfp ,\mabfq}_{E,(\omega )}(\rr d)$
($W^{\mabfp ,\mabfq}_{E,(\omega )}(\rr d)$)
consist of all $f\in \Sigma _1'(\rr d)$ such that
$\nm f{M^{\mabfp ,\mabfq}_{E,(\omega )}}$ 
($\nm f{W^{\mabfp ,\mabfq}_{E,(\omega )}}$) is finite.
\end{defn}

\par

The theory of modulation spaces has developed in different ways since they
were introduced in \cite{F1} by Feichtinger. (Cf. e.{\,}g. \cite{Fei5,GaSa,Gc2,Toft15}.)
For example, let $\mabfp$, $\mabfq$, $E$, $\omega$ and $v$ be the same
as in Definition \ref{Def:ExtClassModSpaces},
and let $L^{\mabfp ,\mabfq}_E(\rr {2d})$ and $r=\min (1,\mabfp ,\mabfq)$.
Then $M^{\mabfp ,\mabfq}_{E,(\omega )}(\rr d)$
is a quasi-Banach space. Moreover, $f\in M^{\mabfp ,\mabfq}_{E,(\omega )}(\rr d)$
if and only if $V_\phi f \cdot \omega \in L^{\mabfp ,\mabfq}_{E}(\rr {2d})$,
and different choices of $\phi$ give rise to equivalent quasi-norms in
Definition \ref{Def:ExtClassModSpaces}.
We also note that
$$
\Sigma _1(\rr d) \subseteq M^{\mabfp ,\mabfq}_{E,(\omega )}(\rr d) \subseteq \Sigma _1'(\rr d).
$$
Similar facts hold for the space $W^{\mabfp ,\mabfq}_{E,(\omega )}(\rr d)$.
(Cf. \cite{GaSa,Toft15}.)

\par

\begin{defn}\label{Def:WienerSpace}
Let $\mabfp ,\mabfr \in (0,\infty ]^d$, 
$\omega _0\in \mascP _E(\rr d)$, $\omega \in \mascP _E(\rr {2d})$,
$\phi \in \Sigma _1(\rr d)\setminus 0$, 
$E\subseteq \rr d$ be an ordered basis, and let $Q_E$ be the closed parallelepiped
spanned by $E$. Also let
$f$ and $F$ be measurable on $\rr d$ respective
$\rr {2d}$, and let $F_\omega =F\cdot \omega$.
Then $\nm f{\sfW ^{\mabfr}_{E}(\omega _0,\ell ^{\mabfp}_{E} )}$ is given by
\begin{align*}
\nm f{\sfW ^{\mabfr}_{E}(\omega _0,\ell ^{\mabfp}_{E} )}
&\equiv
\nm {h_{E,\omega _0,\mabfr ,f}}{\ell ^{\mabfp}_{E}(\Lambda _E)},
\\[1ex]
h_{E,\omega _0,\mabfr ,f} (j)
&=
\nm {f} {L_E^{\mabfr}(j+Q_E)}\omega _0(j),
\quad
j\in \Lambda _E.
\end{align*}
The set $\sfW ^{\mabfr}_{E}(\omega ,\ell ^{\mabfp}_{E} )$ consists of all measurable
$f$ on $\rr d$ such that $\nm f{\sfW ^{\mabfr}_{E}(\omega _0,\ell ^{\mabfp}_{E} )}<\infty$.
\end{defn}

\par

We observe that $\sfW ^{\mabfr}_{E}(\omega _0,\ell _{E}^{\mabfp})$
is equal to $W(L^r,L^{\mabfp}_{(\omega _0)})$ in \cite{FG1,Gc2,GaSa,Rau1,Rau2}
when $E$ is the standard basis. In particular,
$\sfW ^{\mabfr}_{E}(\omega _0,\ell _{E}^{\mabfp})$ is related to so-called coorbit spaces.
(See \cite{Fe0,FG1,FG2,FeLu,Rau1,Rau2}.)

%
%
\par

\begin{rem}\label{Rem:PerWienerSpace}
Let $\mabfp, $$\mabfq$, $\omega _0$, $\omega$, $E$,
$f$ and $F$
be the same as in Definition \ref{Def:WienerSpace}.
Evidently, by using the fact that $\omega _0$ is
$v_0$-moderate for some $v_0$,
it follows that
\begin{align*}
\nm {f\cdot \omega _0}{\sfW ^{\mabfq}_{E}(1,\ell ^{\mabfp}_{E})}
&\asymp
\nm {f}{\sfW ^{\mabfq}_{E}(\omega _0,\ell ^{\mabfp}_{E} )}
\end{align*}
for $k=1,2$.
\end{rem}

\par

\begin{rem}\label{Rem:WienerSpace}
For the spaces in Definition \ref{Def:WienerSpace}
we set $\sfW ^{q_0,\mabfr _0} =  \sfW ^{\mabfr}$, when
\begin{alignat*}{2}
\mabfr _0 &=(r_1,\dots ,r_d)\in (0,\infty ]^d,&
\quad \text{and}\quad
\mabfr &= (q_0,\dots ,q_0,r_1,\dots ,r_d)\in (0,\infty ]^{2d},
\end{alignat*}
and similarly for other types of exponents and for the spaces
in Definition \ref{Def:ExtClassModSpaces}. (See also Remark
\ref{Rem:LebExpIdent}.)
We also set
$$
M^{\infty ,\mabfq}_{E,(\omega )}
=
M^{\infty ,\mabfq}_{E_2,(\omega )}
\quad \text{and}\quad
W^{\infty,\mabfq}_{E,(\omega )}
=
W^{\infty,\mabfq}_{E_2,(\omega )}
$$
when $E_1,E_2$ are ordered bases of $\rr d$ and
$E=E_1\times E_2$, for spaces in Definition \ref{Def:ExtClassModSpaces},
since these spaces are independent of $E_1$.
\end{rem}

\par

The following result is essential when characterizing elements
in modulation spaces in terms of estimates of their
Zak transforms. We omit the proof since the result is
a consequence of \cite[Proposition 2.6]{Toft26}.

\par

\begin{prop}\label{Prop:SpecCaseWienerRel1}
Let $E_0$ be a basis for $\rr d$, $E_0'$ be its dual basis,
$E=E_0\times E_0'$, $\mabfq ,\mabfr \in (0,\infty ]^{d}$,
$\omega _0\in \mascP _E(\rr {d})$ and
$\omega (x,\xi )=\omega _0(\xi )$.
Then
\begin{alignat*}{1}
\nm f{M^{\infty ,\mabfq}_{E,(\omega )}}
&\asymp
\nm {\fy _{f,\omega ,E,\mabfr}}{L^{\mabfq}_{E_0'}}
\quad \text{and}\quad
\nm f{W^{\infty ,\mabfq}_{E,(\omega )}}
\asymp
\sup _{j\in \Lambda _E}
\left (
\nm {\psi _{f,\omega ,\mabfq}}{L^{\mabfr}(j+Q_E)}
\right ),
\intertext{where}
\fy _{f,\omega ,E,\mabfr}(\xi )
&\equiv
\sup _{j\in \Lambda _E}
\left (
\nm {V_\phi f(\cdo ,\xi )\omega (\cdo ,\xi )}{L^{\mabfr}_E(j+Q_E)}
\right )
\intertext{and}
\psi _{f,\omega ,\mabfq} (x)
&\equiv
\nm {V_\phi f(x,\cdo )\omega (x,\cdo )}{L^{\mabfq}_{E_0'}}.
\end{alignat*}
\end{prop}

\par

The next result is a restatement of Proposition 1.15$'$ in \cite{Toft26}.
The proof is therefore omitted. Here
\begin{equation}\label{Eq:SubMultMod}
(\Theta _\rho v)(x,\xi )=v(x,\xi )\eabs {x,\xi}^\rho ,
\quad \text{where}\quad
\rho \ge 2d\left ( \frac 1r-1 \right ).
\end{equation}

\par

\begin{prop}\label{Prop:WienerEquiv}
Let $E$ be a phase split basis for $\rr {2d}$, $\mabfp ,\mabfr \in (0,\infty ]^{2d}$,
$r\in (0,\min (1,\mabfp )]$, $\omega ,v\in \mascP _E(\rr {2d})$
be such that $\omega$ is $v$-moderate, $\rho$ and $\Theta _\rho v$ $\rho$ be
as in \eqref{Eq:SubMultMod} with strict inequality when $r<1$, and let $\phi _1,\phi _2\in
M^1_{(\Theta _\rho v)} (\rr d)\setminus 0$.
Then
$$
\nm f{M^{\mabfp}_{E,(\omega )}}
\asymp
\nm {V_{\phi _1} f}{L^\mabfp _{E,(\omega )}}
\asymp 
\nm {V_{\phi _2} f}{\sfW ^{\mabfr} _{E}(\omega ,\ell ^\mabfp _{E})},\quad f\in
\maclS '_{1/2}(\rr d).
$$
In particular, if $f\in \maclS _{1/2}'(\rr d)$, then
$$
f\in M^{\mabfp}_{E,(\omega )}(\rr {2d})
\quad \Leftrightarrow \quad
V_{\phi _1} f \in L^\mabfp _{E,(\omega )}(\rr {2d})
\quad \Leftrightarrow \quad
V_{\phi _2} f \in \sfW ^{\mabfr} _{E}(\omega ,\ell ^\mabfp _{E}
(\Lambda _E)).
$$
\end{prop}

\par

\subsection{Classes of periodic elements}

\par

Let $s,\sigma \in\mathbf{R}_{+}$ be such that $s+t\geq 1$,
$f\in (\mathcal S_{s}^{\sigma})'(\rr d)$, 
$E$ be a basis of $\rr d$ and let $E_0\subseteq E$.
Then $f$ is called \emph{$E_0$-periodic} if $f(x+y)=f(x)$ 
for every $x\in \rr d$ and $y\in E_0$.

\par

We note that for any $E$-periodic function
$f\in C^{\infty}(\rr d)$, we have 
\begin{align}
f &= \sum_{\alpha\in \Lambda ' _E} c(f,{\alpha})e^{i\scal \cdo \alpha},
\label{Eq:Expan2}
\intertext{where $c(f,{\alpha})$ are the Fourier coefficients given by}
c(f,{\alpha}) &\equiv \vert Q_E \vert ^{-1}
(f,e^{i\scal \cdo \alpha})_{L^{2}(E)}.\notag
\end{align}

\par

For any $s\ge 0$ and basis $E\subseteq \rr d$ we let
$\maclE ^{\sigma ;0}_{E}(\rr d)$
and $\maclE ^{\sigma}_{E}(\rr d)$ be the sets of all $E$-periodic 
elements
in $\maclE ^{\sigma ;0}(\rr d)$ and in $\maclE ^{\sigma}(\rr d)$, respectively.

\par

Let $s,s_0,\sigma ,\sigma _0>0$ be such that $s+\sigma \ge 1$,
$s_0+\sigma _0\ge 1$ and $(s_0,\sigma _0)\neq (\frac 12,\frac 12)$. Then
we recall that the duals
$(\maclE ^\sigma _E)'(\rr d)$ and $(\maclE ^{\sigma _0;0}_E)'(\rr d)$ of
$\maclE ^\sigma _E(\rr d)$ and $\maclE ^{\sigma _0;0}_E(\rr d)$, respectively,
can be identified with the $E$-periodic elements in $(\maclS _s^\sigma)'(\rr d)$
and $(\Sigma _{s_0}^{\sigma _0})'(\rr d)$ respectively via unique extension of the form
\begin{align*}
(f,\phi)_{E} &= \sum_{\alpha \in \Lambda '_{E}} c(f,{\alpha})
\overline{c(\phi ,{\alpha})}
\end{align*}
on $\maclE ^{\sigma _0;0}_E (\rr d)\times \maclE ^{\sigma _0;0}_E (\rr d)$.
We also let $(\maclE ^0_E)'(\rr d)$ be the set of all formal expansions
in \eqref{Eq:Expan2} and $\maclE ^0_E(\rr d)$ be the set of
all formal expansions in \eqref{Eq:Expan2} such that at most finite
numbers of $c(f,{\alpha})$ are non-zero (cf. \cite{ToNa}). We refer to
\cite{Pil2,ToNa} for more characterizations of $\maclE ^{\sigma}_E$,
$\maclE ^{\sigma ;0}_E$ and their duals.

\par

The following definition takes care of spaces of formal
expansions \eqref{Eq:Expan2} with coefficients
obeying specific quasi-norm estimates.

\par

\begin{defn}\label{Def:PerQuasiBanachSpaces}
Let $E$ be a basis of $\rr d$, $\mascB$ be a quasi-Banach
space  continuously embedded in $\ell _0' (\Lambda ' _E)$
and let $\omega _0$ be a weight on $\rr d$. Then
$\maclE _E(\omega _0,\mascB )$ consists of all 
$f\in(\maclE^{0}_E)'(\rr d)$ such that 
\begin{equation*}
\nm f{\maclE _E(\omega _0,\mascB)}
\equiv
\nm {\{ c(f,\alpha)\omega _0(\alpha) \}
_{\alpha \in \Lambda '_E}} {\mascB}
\end{equation*}
is finite.
\end{defn}

\par

The next result is a reformulation of Proposition 1.18$'$ in
\cite{Toft26}. The proof is therefore omitted.

\par

\begin{prop}\label{Prop:PerMod}
Let $E$ be an ordered basis of $\rr d$, $E=E_0\times E_0'$,
$\mabfq ,\mabfr \in (0,\infty ]^d$, 
$\omega _0\in \mascP _E(\rr d)$ and let $\omega (x,\xi )= \omega _0(\xi )$,
$x,\xi \in \rr d$. Then
\begin{equation*}
\maclE _{E}(\omega _0,\ell _{E_0'}^{\mabfq}(\Lambda _{E_0'}))
=
M_{E,(\omega )}^{\infty ,\mabfq}(\rr d)\bigcap (\maclE ^0_E)'(\rr d)
=
W_{E,(\omega )}^{\infty ,\mabfq}(\rr d)\bigcap (\maclE ^0_E)'(\rr d).
\end{equation*}
\end{prop}

\par

\begin{rem}\label{Rem:PerMod}
Let $E_0$, $\mabfq$, $\mabfr$
and $\omega$ be the same as in Proposition
\ref{Prop:PerMod}. Also let $v_0\in \mascP _E(\rr d)$ be such that
$\omega _0$ is $v_0$-moderate, $v(x,\xi )=v_0(\xi )$, $\phi
\in \Sigma _1(\rr d)\setminus 0$,
$r_0\le \min (\mabfr )$, $f\in (\maclE ^0_{E})'(\rr d)$
with Fourier series expansion \eqref{Eq:Expan2},
$$
E_1=E_0\times E_0' = \{ e_1,\dots ,e_{2d}\} 
\quad \text{and}\quad
E_2=\{ e_{d+1},\dots ,e_{2d},e_1,\dots ,e_d \} .
$$
Then it follows from Proposition 2.7 and Remark 2.8 in \cite{Toft26}
that 
\begin{equation}\label{Eq:ModPerNormEquiv3}
\nm {V_\phi f \cdot \omega}{L^{\mabfr ,\mabfq}_{E_1}(Q(E_0)\times \rr d)}
\asymp
\nm {V_\phi f \cdot \omega}{L^{\mabfq ,\mabfr}_{E_2}(\rr d \times Q(E_0))}
\asymp
\nm {c(f,\cdo )}{\ell ^{\mabfq}_{E_0',(\omega _0)}}.
\end{equation}
\end{rem}

\par

\subsection{The Zak transform}

\par

For any ordered basis $E$ of $\rr d$ and $f\in \mascS (\rr d)$,
the \emph{Zak transform} is defined by
\begin{equation}\label{Eq:DefZak}
(Z_E f)(x,\xi ) \equiv \sum _{j\in \Lambda _E}
f(x-j)e^{i\scal j\xi}
\end{equation}

\par

Several properties for the Zak transform can be found in \cite{Gc2}. 
For example, by
the definition it follows that $Z_E$ is continuous from $\mascS 
(\rr d)$ to the set of all smooth functions
on $\rr {2d}$ which are
bounded together with all their derivatives. It is also clear
that $Z_E f$ is quasi-periodic
of order $E$. Here, if $F$ is a function or an ultra-distribution,
then $F$ is called quasi-periodic of order $E$, when
\begin{equation}\label{Eq:QuasiPer}
\begin{alignedat}{1}
F(x+k,\xi ) = e^{i\scal k\xi} F(x,\xi )
\quad \text{and}\quad
F(x,\xi +\kappa ) &= F(x,\xi ),
\\[1ex]
k &\in \Lambda _E,\ 
\kappa \in \Lambda _E'.
\end{alignedat}
\end{equation}
By interpreting \eqref{Eq:DefZak} as a Fourier series in
the $\xi$ variable, we regain $f(x)$ as the zero order Fourier
coefficient, which is evaluated by
\begin{equation}\label{Eq:ZakTransfInv}
f(x) = (Z_E^{-1}F)(x) = |Q_{E'}|^{-1}\int _{Q_{E'}}F(x,\xi )\, d\xi 
\end{equation}

\par

For conveniency we
set $Z_1=Z_E$ when $E$ is the standard basis of $\rr d$,
and recall the following important mapping
properties on $L^2(\rr d)$ and $\mascS (\rr d)$. 

\par

\begin{prop}\label{Prop:ZakTransfBasicMaps}
Let $E$ be an ordered basis of $\rr d$. Then the operator $Z_E$
is homeomorphic from $L^2(\rr d)$
to the set of all quasi-periodic elements of order $E$ in
$L^2_{loc}(\rr {2d})$, and
\begin{equation}\tag*{(\ref{Eq:ZakonL2})$'$}
\nm {Z_E f}{L^2(Q _{E\times E'})} = |Q_{E'}|^{\frac 12}
\nm f{L^2},
\qquad f\in L^2(\rr d).
\end{equation}
\end{prop}

\par

\begin{proof}
Let $T_E$ be as in Remark \ref{Rem:TEMap}.
By straight-forward computations it follows that
\begin{equation}\label{Eq:ZakTransfers}
Z_Ef(x,\xi ) = (Z_1f_E)(T_E^{-1}x,T_E^*\xi ),\quad
f_E=f\circ T_E.
\end{equation}
The assertion now follows from \eqref{Eq:ZakonL2},
\eqref{Eq:ZakTransfers} and
suitable changes of variables in the involved integrals.
The details are left for the reader.
\end{proof}

\par

\section{Zak transform on Gelfand-Shilov spaces,
Lebesgue spaces and modulation spaces}
\label{sec2}

\par

In this section we deduce characterizations of Lebesgue spaces, modulation
spaces, and Gelfand-Shilov spaces and their distribution spaces in terms of
suitable estimates of the Zak transforms of the involved elements. The
characterizations on modulation spaces are related to results given in
\cite{Ti,TiHe}.

\par

\subsection{Spaces of quasi-periodic functions and
distributions}\label{subsec2.1}

\par

Since quasi-periodic functions depend on the phase space
variable $(x,\xi )\in \rr {2d}$, it is suitable that the
Gevrey regularity with respect to $x\rr d$ for such functions
might be different to the Gevrey regularity with respect to $\xi \in \rr d$.
We therefore consider two parameters analogies of $\maclE ^s$ and
$\maclE ^{s;0}$, where the parameter $s$ is replaced by the pair
$\sigma, s$. More preceisely,
for any compact $K\subseteq \rr {2d}$ and $s,\sigma \ge 0$,
$\maclE ^{\sigma ,s}(K)$ ($\maclE ^{\sigma ,s;0}(K)$)
is the set of all $F\in C^\infty (K)$ such that
$$
\sup _{\alpha ,\beta \in \nn d}\sup _{x,\xi \in \rr d}
\left (
\frac {|\partial _x^\alpha \partial _\xi F(x,\xi )}
{h^{|\alpha +\beta |}\alpha !^\sigma \beta !^s}
\right )
$$
is finite for some $h>0$ (for every $h>0$). The two parameter
Gevrey classes,
$\maclE ^{\sigma ,s}(\rr d)$ and $\maclE ^{\sigma ,s;0}(\rr d)$,
are the projective limits of
$\maclE ^{\sigma ,s}(K_j)$ respective
$\maclE ^{\sigma ,s;0}(K_j)$, when $\{ K_j\} _{j\ge 1}$
is an exhausted set of compact subsets of $\rr {2d}$. Furthermore we let
\begin{alignat*}{3}
& \maclE ^{\sigma ,s;0}_{Z,E}(\rr {2d}), &
\qquad
& \maclE ^{\sigma ,s}_{Z,E}(\rr {2d}), &
\qquad
& C^\infty _{Z,E}(\rr {2d}),
\\[1ex]
& \mascS '_{Z,E}(\rr {2d}), &
\qquad
& (\maclE ^{\sigma ,s}_{Z,E})'(\rr {2d}), &
\qquad 
& (\maclE ^{\sigma ,s;0}_{Z,E})'(\rr {2d})
\intertext{be the set of all quasi-periodic elements in}
&\maclE ^{\sigma ,s;0}(\rr {2d}), &
\qquad
& \maclE ^{\sigma ,s}(\rr {2d}), &
\qquad
& C^\infty (\rr {2d}),
\\[1ex]
& \mascS ' (\rr {2d}), &
\qquad
& (\maclS _{s,\sigma}^{\sigma ,s})'(\rr {2d}), &
\qquad
& (\Sigma _{s,\sigma}^{\sigma ,s})'(\rr {2d}),
\end{alignat*}
respectively, with respect to the ordered
basis $E$ on $\rr d$. For conveniency we also set
$$
\maclE ^{\sigma ,s;0}_{Z}=\maclE ^{\sigma ,s;0}_{Z,E} \quad
\text{and}\quad \maclE ^{\sigma ,s}_{Z}=\maclE ^{\sigma ,s}_{Z,E},
$$
when $E$ is the standard basis of $\rr d$.

\par

Next we introduce spaces of quasi-periodic functions and distributions
which correspond to Lebesgue spaces and modulation spaces. We let
$L^p_{Z,E}(\rr {2d})$ for $p\in (0,\infty ]$ be the set of all quasi-periodic
measurable functions $F$ on $\rr {2d}$ with respect to the ordered
basis $E$ such that
$$
\nm F{L^p_{Z,E}}\equiv |Q_{E'}|^{-\frac 1p}
\nm F{L^p(Q _{E\times E'})}
$$
is finite. Evidently, we may identify $L^p_{Z,E}(\rr {2d})$ by
$L^p(Q _{E\times E'})$, and the scalar product on
$L^2_{Z,E}(\rr {2d})$ is given by
\begin{equation}\label{Eq:QuasiPerScalarProd}
(F,G)_{Z,E}\equiv (F,G)_{L^2_{Z,E}(\rr {2d})}
=
|Q_{E'}|^{-1}(F,G)_{L^2(Q_{E\times E'})}
\end{equation}
when $F,G\in L^2_{Z,E}(\rr {2d})$.

\par

Let $\mabfp ,\mabfr \in (0,\infty ]^{2d}$, $\omega \in \mascP _E(\rr {4d})$
and $\Phi \in \Sigma _1(\rr {2d})\setminus 0$ be fixed, and let $E_0$ be
an ordered basis in $\rr {d}$. Then set
\begin{align}
H_{1,F,\omega}(x,\xi ,y,\eta ) 
&=
|V_\Phi F(x,\xi ,\eta ,y)\omega (x,\xi ,\eta ,y)|,\label{Eq:STFTZakPartNorm1}
\intertext{and}
H_{2,F,\omega ,E,E_0,\mabfp}(x,\xi )
&\equiv
\nm {H_{1,F,\omega}(x,\xi ,\cdo )}{L^{\mabfp}_{E\times E_0'}(\rr {2d})},
\label{Eq:STFTZakPartNorm2}
\end{align}
when $F$ is a quasi-periodic Gelfand-Shilov
distribution with respect to the ordered basis $E$.
We let $W_{Z,E,E_0,(\omega )}^{\mabfr , \mabfp}(\rr {2d})$ be the set of all
quasi-periodic Gelfand-Shilov distributions $F$ with respect to the ordered
basis $E$ such that
$$
\nm F{W_{Z,E,E_0,(\omega )}^{\mabfr , \mabfp}}
\equiv
\nm {H_{2,F,\omega ,E,E_0,\mabfp}}{L^{\mabfr}_{E\times E'}(Q _{E\times E'})}
$$
is finite. We also let the topology of $W_{Z,E,E_0,(\omega )}^{\mabfr , \mabfp}(\rr {2d})$
be induced by the quasi-norm $\nm \cdo{W_{Z,E,E_0,(\omega )}^{\mabfr , \mabfp}}$.
Usually we assume that $\omega$ is given by
\begin{equation}\label{Eq:ZakModulationWeight}
\omega (x,\xi ,\eta ,y)=\omega _0(x-y,\eta ),
\end{equation}
for some $\omega _0\in \mascP _E(\rr {2d})$.

\par

\subsection{The Zak transform on test function spaces
and their distribution spaces}

\par

For the classical spaces $\mascS (\rr d)$ and its distribution
space $\mascS '(\rr d)$ we have the following.

\par

\begin{thm}\label{Thm:ZakClassical}
Let $E$ be an ordered basis of $\rr d$. Then the
following is true:
\begin{enumerate}
\item The operator $Z_E$ is a homeomorphism from
$\mascS (\rr d)$ to $C^\infty _{Z,E}(\rr {2d})$;

\vrum

\item The operator $Z_E$ from $\mascS (\rr d)$
to $C^\infty (\rr {2d})$ is uniquely extendable to a
homeomorphism from $\mascS '(\rr d)$ to $\mascS '_{Z,E}(\rr {2d})$.
\end{enumerate}
\end{thm}

\par

The assertion (1) in Theorem \ref{Thm:ZakClassical}
follows from \eqref{Eq:ZakTransfers} and
\cite[Theorem 8.2.5]{Gc2},
and (2) in the same theorem follows by similar arguments
as in the proof of Theorem \ref{Thm:ZakDist} below.
The verifications of Theorem \ref{Thm:ZakClassical}
are therefore left for the reader.

\par

The analogous result of the previous theorem for 
Gelfand-Shilov functions and their distributions, are given
in Theorems \ref{Thm:ZakTestFunctions} and \ref{Thm:ZakDist} below.

\par

\begin{thm}\label{Thm:ZakTestFunctions}
Let $s,\sigma >0$ and $E$ be an ordered basis.
Then the operator $Z_E$ from $\mascS (\rr d)$ to
$C^\infty (\rr {2d})$
restricts to a homeomorphism from $\maclS _s^\sigma (\rr d)$
to $\maclE ^{\sigma ,s}_{Z,E}(\rr {2d})$.

\par

The same holds true with $\Sigma _s^\sigma (\rr d)$
and $\maclE ^{\sigma ,s;0}_{Z,E}(\rr {2d})$
in place of $\maclS _s^\sigma (\rr d)$
and $\maclE ^{\sigma ,s}_{Z,E}(\rr {2d})$, respectively
at each occurrence.
\end{thm}

\par

By the previous result and the facts that $\maclS _s^\sigma (\rr d)$
is trivially equal to $\{ 0\}$ when $s+\sigma <1$, and $\Sigma _s^\sigma (\rr d)$
is trivial when $s+\sigma <1$ or $s=\sigma =\frac 12$, we get the
following.

\par

\begin{cor}\label{Cor:ZakTestFunctions}
Let $s,\sigma >0$ and $E$ be an ordered basis.
Then the following is true:
\begin{enumerate}
\item if $s+\sigma <1$ and
$F\in \maclE ^{\sigma ,s}_{Z,E}(\rr {2d})$, then $F=0$;

\vrum

\item if $s+\sigma <1$ or $s=\sigma =\frac 12$ and
$F\in \maclE ^{\sigma ,s;0}_{Z,E}(\rr {2d})$, then $F=0$.
\end{enumerate}
\end{cor}

\par

\begin{rem}\label{Rem:ZakTestFunctions}
Let $\sigma \ge 0$, $E$ be a basis of $\rr d$ and let
$f$ be an $E$-periodic distribution on $\rr d$.
Then $f\in \maclE ^\sigma (\rr d)$, if and only if
its Fourier coefficients $c(f,\alpha )$ in
\eqref{Eq:Expan2} satisfies
$$
|c(f,\alpha )|\lesssim e^{-r|\alpha |^{\frac 1\sigma}}
$$
for some $r>0$. In particular, $\maclE ^\sigma _E(\rr d)\neq \{ 0\}$
for every $\sigma \ge 0$ (cf. \cite{Pil2,ToNa}).
Consequently, the conclusions in Corollary
\ref{Cor:ZakTestFunctions} are not true
for periodic functions in place of quasi-periodic
functions. 
\end{rem}

\par

\begin{thm}\label{Thm:ZakDist}
Let $s,\sigma >0$ and
$E$ be an ordered basis of $\rr d$. Then the
operator $Z_E$ from $\mascS (\rr d)$ to
$C^\infty (\rr {2d})$
extends uniquely to a homeomorphism from $(\maclS
_s^\sigma)'(\rr d)$
to $(\maclE ^{\sigma ,s}_{Z,E})'(\rr {2d})$.

\par

The same holds true with $(\Sigma _s^\sigma)'(\rr d)$
and $(\maclE ^{\sigma ,s;0}_{Z,E})'(\rr {2d})$
in place of $(\maclS _s^\sigma)'(\rr d)$
and $(\maclE ^{\sigma ,s}_{Z,E})'(\rr {2d})$, respectively
at each occurrence.
\end{thm}

\par

We need the following lemma for the proof of Theorem
\ref{Thm:ZakTestFunctions}.

\par

\begin{lemma}\label{Lemma:ZakTestFunctions}
Let $r,s>0$. Then there is a constant $h>0$ such that
$$
|t|^\beta e^{-r\, |t|^{\frac 1s}} \lesssim h^\beta \beta !^s,
\qquad t\in \mathbf R,\ \beta \in \mathbf N.
$$
\end{lemma}

\par

\begin{proof}
By reasons of symmetry and Stirling's formula, the result follows if we prove
$$
t^\tau e^{-r\, t^{\frac 1s}} \lesssim h^\tau \tau ^{s\tau},
\qquad 0\le t,\tau \in \mathbf R
$$
for some $h>0$. By taking the logaritm it follows that we need to prove that
for some constant $C>0$,
$$
g(\tau ) = -rt^{\frac 1s} + \tau \log t - C\tau -s\tau \log \tau 
$$
is bounded from above by a constant which is independent of $t,\tau >0$.

\par

By differentiation and analysing the sign of $g'(\tau)$, it follows that $g(\tau)$
has a global maximum for
$$
\tau _0=e^{-1-\frac Cs} t^{\frac 1s}
$$
with value
$$
g(\tau _0) = t^{\frac 1s}(-r +seh^{-\frac 1s})
$$
By choosing
$$
h>
\left (
\frac r{se}
\right )^s
$$
it follows that $g(\tau )$ is negative, giving the result.
\end{proof}

\par

\begin{proof}[Proof of Theorem \ref{Thm:ZakTestFunctions}]
Let $T_E$ be the same as in \eqref{Eq:ZakTransfers}.
Then the map $F(x,\xi )\mapsto F(T_E^{-1}x,T_E^*\xi )$
maps quasi-periodc elements of order $E$ to
quasi-periodic elements with respect to the
standard basis. Since $f\mapsto f\circ T$ maps
$E$-periodic elements to $1$-periodic functions,
it follows from these observations and
\eqref{Eq:ZakTransfers}
that it suffices to prove the result when $E$
is the standard basis.

\par

The assertion (1) is the same as Theorem 8.2.5 in \cite{Gc2}.

\par

In order to prove (2) we shall follow the proof of
Theorem 8.2.5 in \cite{Gc2}. In fact,
assume first that $f\in \maclS _s^\sigma (\rr d)$, $x\in  k_0+Q_{d,1}$ for some
fixed $k_0\in \zz d$, and let $F=Z_1 f$. Then
$$
|\partial _x^\alpha f(x-k)|\le C h^\alpha e^{-2r (|k_1|^{\frac 1s}+\cdots +|k_d|^{\frac 1s})}
\alpha !^\sigma , \quad x\in k_0+Q_{d,1},
$$
for some positive constant $C$ which only depends on $k_0$ and some positive constants
$h>0$ and $r>0$ which are independent of $x$, $k_0$, $k$ and $\alpha$.

\par

The series in \eqref{Eq:DefZak} is absolutely convergent together with all its derivatives.
This gives
\begin{multline*}
|(\partial _x^\alpha \partial _\xi ^\beta F)(x,\xi )| 
\le
\sum _{k\in \zz d}|f^{(\alpha )}(x-k)|
k^{\beta}
\\[1ex]
\lesssim
h_1^{|\alpha +\beta |}\sum _{k\in \zz d}
e^{-2r (|k_1|^{\frac 1s}+\cdots +|k_d|^{\frac 1s})}|k^\beta |
\\[1ex]
=
h_1^{|\alpha +\beta |}
\prod _{j=1}^d
\left (
\sum _{k_j\in \mathbf Z}e^{-r|k_j|^{\frac 1s}} \left (e^{-r|k_j|^{\frac 1s}} |k_j|^{\beta _j}\right )
\right ) ,
\end{multline*}
for some constant $h_1>0$. By Lemma \ref{Lemma:ZakTestFunctions}
it follows that
$$
e^{-r|k_j|^{\frac 1s}} |k_j|^{\beta _j} \lesssim h^{\beta _j}\beta _j!^s
$$
for some constant $h>0$. A combination of these estimates give
$$
|(\partial _x^\alpha \partial _\xi ^\beta F)(x,\xi )| 
\lesssim
h^{|\alpha +\beta |}\alpha !^\sigma \beta !^s,
$$
and it follows that $F\in \maclE ^{\sigma ,s}(\rr {2d})$.
This shows that
$Z_1$ is continuous from $\maclS _s^\sigma (\rr d)$
to $\maclE ^{\sigma ,s}_Z(\rr {2d})$.

\par

Next we show that any $F$
in $\maclE ^{\sigma ,s}_Z(\rr {2d})$
is the Zak transform of an element in
$\maclS _s^\sigma (\rr d)$. By Theorem 8.2.5
in \cite{Gc2} it follows that $F=Z_1 f$ when
$$
f(x) = \int _{Q_{d,2\pi}}F(x,\xi )\, d\xi .
$$
We need to prove that $f\in \maclS _s^\sigma (\rr d)$.

\par

Since $k\mapsto f(x-k)$ is the Fourier coefficient
of order $k$ for the function
$\xi \mapsto F(x,\xi)$, we have
$$
f(x-k)= \int _{Q_{d,2\pi}} F(x,\xi )e^{-i\scal k\xi}
\, d\xi .
$$
By applying the operator $k^\alpha
\partial _x^\beta$ and integrating by parts we get
\begin{multline*}
k^\alpha (\partial _x^\beta f)(x-k) = 
\int _{Q_{d,2\pi}} (\partial _x^\beta F)(x,\xi )
k^\alpha e^{-i\scal k\xi}\, d\xi 
\\[1ex]
= (-1)^{|\beta |}i^{|\alpha |}
\int _{Q_{d,2\pi}} (\partial _x^\beta
\partial _\xi ^\alpha F)(x,\xi )
e^{-i\scal k\xi}\, d\xi .
\end{multline*}
This gives
\begin{multline*}
\sup _{x\in \rr d} |x^\alpha f^{(\beta )}(x)|
\asymp
\sup _{k\in \zz d} \sup _{x\in Q_{d,\rho}}
|k^\alpha f^{(\beta )}(x-k)|
\\[1ex]
\lesssim
\nm{\partial _x^\beta \partial _\xi ^\alpha F}
{L^\infty (Q_{d,\rho}\times Q_{d,2\pi})}
\lesssim
h^{|\alpha +\beta |}\alpha !^s\beta !^\sigma ,
\end{multline*}
which is the same as $f\in \maclS _s^\sigma (\rr d)$.
This gives (2). The assertion (3) follows by similar
arguments and is left for the reader.
\end{proof}

\par

For the proof of Theorem \ref{Thm:ZakDist}
we need the following lemma on tensor product
of Gelfand-Shilov distributions.

\par

\begin{lemma}\label{Lemma:GSTensorProducts}
Let $s_j,\sigma _j>0$ and $f_j\in (\maclS _{s_j}^
{\sigma _j})'(\rr {d_j})$, $j=1,2$. Then there
is a unique $f\in
(\maclS _{s_1,s_2}^{\sigma _1,\sigma _2})'(\rr {d_1+d_2})$
such that
$$
\scal f{\fy _1\otimes \fy _2}
=\scal {f_1}{\fy _1}\scal {f_2}{\fy _2}
,\qquad \fy _j\in \maclS _{s_j}^{\sigma _j}(\rr {d_j}),\ j=1,2.
$$
Moreover, if $\fy \in \maclS _{s_1,s_2}
^{\sigma _1,\sigma _2}(\rr {d_1+d_2})$,
$$
\fy _1 (x_1)= \scal {f_2}{\fy (x_1,\cdo )}
\quad \text{and}\quad
\fy _2 (x_2)= \scal {f_1}{\fy (\cdo ,x_2)},
$$
then
$$
\scal f{\fy} = \scal {f_1}{\fy _1} = \scal {f_2}{\fy _2}.
$$

\par

The same holds true with $\Sigma _{s_j}^{\sigma _j}$,
$\Sigma _{s_1,s_2}
^{\sigma _1,\sigma _2}$, $(\Sigma _{s_j}^{\sigma _j})'$ and
$(\Sigma _{s_1,s_2}^{\sigma _1,\sigma _2})'$ in place of
$\maclS _{s_j}^{\sigma _j}$, $\maclS _{s_1,s_2}^{\sigma 
_1,\sigma _2}$,
$(\maclS _{s_j}^{\sigma _j})'$ and $(\maclS 
_{s_1,s_2}^{\sigma _1,\sigma _2})'$,
respectively, $j=1,2$.
\end{lemma}

Lemma \ref{Lemma:GSTensorProducts}
is essentially a restatement of Theorem 2.4
in \cite{Toft24}. The proof is therefore omitted.

\par

\begin{rem}\label{Rem:TensorUniqueness}
We notice that the uniqueness assertions in
Lemma \ref{Lemma:GSTensorProducts} is an
immediate consequence of \cite[Lemma 2.3]{Toft24}
which asserts that if
$f\in
(\maclS _{s_1,s_2}^{\sigma _1,\sigma _2})'(\rr {d_1+d_2})$
($f\in
(\Sigma _{s_1,s_2}^{\sigma _1,\sigma _2})'(\rr {d_1+d_2})$)
satisfies
$$
\scal f{\fy _1\otimes \fy _2}=0
$$
for every $\fy _j\in \maclS _{s_j}^{\sigma _j}(\rr {d_j})$
($\fy _j\in \Sigma _{s_j}^{\sigma _j}(\rr {d_j})$),
then $f=0$ (as an element in
$(\maclS _{s_1,s_2}^{\sigma _1,\sigma _2})'(\rr {d_1+d_2})$
($(\Sigma _{s_1,s_2}^{\sigma _1,\sigma _2})'(\rr {d_1+d_2})$)).
\end{rem}

\par

\begin{proof}[Proof of Theorem \ref{Thm:ZakDist}]
By similar arguments as in the proof of Theorem
\ref{Thm:ZakTestFunctions} we may assume that $E$ is the
standard basis for $\rr d$.

\par

We begin to prove (2). Let $\Phi \in \maclS 
_{s,\sigma}^{\sigma ,s}(\rr {2d})$. Then
\begin{gather*}
\mascF _2^{-1}\Phi \in \maclS _s^\sigma (\rr {2d}),\quad 
\mascF _1 (\mascF _2^{-1}\Phi ) \in \maclS _{\sigma ,s}^{s,\sigma}(\rr {2d}),
\\[1ex]
|\mascF _2^{-1}\Phi (x,y)|\lesssim e^{-r_0(|x|^{\frac 1s}+|y|^{\frac 1s})}
\quad \text{and}\quad
|\mascF _1\Phi (\xi ,\eta)|\lesssim e^{-r_0(|\xi |^{\frac 1\sigma }+|\eta |^{\frac 1\sigma })}
\end{gather*}

\par

If $f\in \mascS (\rr d)$, then
\begin{multline}\label{Eq:ZakDistFormula}
\scal {Z_1 f}\Phi = \iint _{\rr {2d}} \left (\sum _{j\in \zz d}f(x-j)e^{i\scal j\xi}
\right )\Phi (x,\xi )\, dxd\xi
\\[1ex]
=
\sum _{j\in \zz d} \left (
\int _{\rr d} f(x) (\mascF _2^{-1}\Phi )(x+j, j) 
\, dx
\right )
\\[1ex]
=
\sum _{j\in \zz d} (\check f *\Phi _2(\cdo ,j))(j),
\end{multline}
where $\Phi _2= \mascF _2^{-1}\Phi $ and $\check f(x)=f(-x)$.

\par

Assume instead that $f\in (\maclS _s^\sigma)'(\rr d)$ is 
arbitrary. We claim that the series
on the right-hand side of \eqref{Eq:ZakDistFormula} converges 
absolutely for every $\Phi$ as above.

\par

In fact, since $f\in (\maclS _s^\sigma)'(\rr d)$, we have
$$
|\scal f\phi |
\lesssim
\nm {e^{r|\cdo |^{\frac 1t}} \phi}{L^\infty}
+
\nm {e^{r|\cdo |^{\frac 1\sigma }} \widehat \phi}{L^\infty}
$$
for every $r>0$, giving that for some $r_0>0$ and $c\ge 1$
we have
\begin{multline*}
| (\check f *\Phi _2(\cdo ,y))(x) |
\lesssim
\nm {e^{r|\cdo |^{\frac 1s}} \Phi _2(x-\cdo ,y) }{L^\infty}
+
\nm {e^{r|\cdo |^{\frac 1\sigma }}
\widehat \Phi (\cdo ,y ) }{L^\infty}
\\[1ex]
\lesssim e^{rc|x|^{\frac 1s}-2r_0|y|^{\frac 1s}/c}.
\end{multline*}
Hence, if $r$ is chosen small enough and $x=y=j$, then
\begin{equation}\label{Eq:ConvGSEst}
| (\check f *\Phi _2(\cdo ,j ))(j) | \lesssim e^{-r_0| j |^{\frac 1s}/c},
\end{equation}
when $(\maclS _s^\sigma)'(\rr d)$. The absolutely
convergence of the series
of the right-hand side of \eqref{Eq:ZakDistFormula}
now follows from \eqref{Eq:ConvGSEst}.

\par

If $f\in (\maclS _s^\sigma)'(\rr d)$, then $Z_1 f$
is defined as the element in
$(\maclS _{s,\sigma}^{\sigma ,s})'(\rr {2d})$, given
by the right-hand side of
\eqref{Eq:ZakDistFormula}. The previous estimates
show that this
definition makes sense, and that the map
$f\mapsto Z_1 f$ is continuous
from from $(\maclS _s^\sigma)'(\rr d)$
to $(\maclE ^{\sigma ,s}_Z)'(\rr {2d})$.
By approximating elements in $(\maclS _s^\sigma )'(\rr d)$
by sequences of elements in $\mascS (\rr d)$, it also
follows that the continuous extension of $Z_1$
to such distribution is unique.

\par

We need to prove that any element in
$(\maclE ^{\sigma ,s}_Z) '(\rr {2d})$ is the Zak transform of an element in
$(\maclS _s^\sigma)'(\rr d)$. Therefore, let
$\fy \in \maclS _s^\sigma (\rr d)$, $F\in 
(\maclE ^{\sigma ,s}_Z)'(\rr {2d})$, and let 
$g_\fy\in (\maclS _\sigma ^s)'(\rr d)$ be defined
by
$$
\scal {g_\fy}{\psi} = \scal F{\fy \otimes \psi},
\qquad \psi \in \maclS _\sigma ^s (\rr d).
$$
Then $g_\fy$ is $2\pi$-periodic, and it follows from
Remark 2.3 and Proposition 2.5
in \cite{ToNa} that if $\phi \in \maclS _\sigma ^s
(\rr d)\setminus 0$, then
\begin{align*}
g_\fy &= \sum _{k\in \zz d}c(g_\fy ,k)e^{i\scal k\xi},
\intertext{where the series converges in
$(\maclS _\sigma ^s)'(\rr d)$, and}
c(g_\fy ,0)
&=
\frac {1}{(2\pi )^d\nm \phi{L^2}^2}\int _{Q_{d,2\pi}}
\left ( \int _{\rr d}
(V_\phi g_\fy )(\eta ,y)\widehat \phi (-y)e^{i\scal y\eta}\, dy
\right ) \, d\eta
\intertext{and}
c(g_\fy ,k)
&=
c(g_\fy e^{-i\scal k{\cdo}},0)
\end{align*}

\par

By straight-forward computations we get
\begin{align}
c(g_\fy ,0) &= 
\nm \phi{L^2}^{-2}
\int _{Q_{d,2\pi}}
h_\fy (\eta)\, d\eta ,
\label{Eq:cg0coeff}
\intertext{where}
h_\fy (\eta ) &= (2\pi )^{-\frac {3d}2}\int _{\rr d}
\scal F{\fy \otimes (\overline{\phi (\cdo -\eta )}
e^{-{i\scal y\cdo}})}
\widehat \phi (-y)e^{i\scal y\eta}\, dy ,
\notag
\end{align}
and it is clear that the map which takes $\fy$ into the 
right-hand side in \eqref{Eq:cg0coeff}
defines a continuous linear form on
$\maclS _s^\sigma (\rr d)$. Hence
$$
c(g_\fy ,0) = \scal f\fy ,\qquad \fy \in
\maclS _s^\sigma (\rr d),
$$
for some $f\in (\maclS _s^\sigma)'(\rr d)$. Furthermore, by
the quasi-periodicity of $F$ we obtain
\begin{multline*}
\scal {g_\fy e^{-i\scal k\cdo}}{\psi}
=
\scal {g_\fy }{\psi e^{-i\scal k\cdo}}
=
\scal F{\fy \otimes (\psi e^{-i\scal k\cdo} )}
\\[1ex]
=
\scal {F(\cdo -(k,0))}{\fy \otimes \psi}
=
\scal {F}{\fy (\cdo +k)\otimes \psi}
=
\scal {g_{\fy (\cdo +k)}}{\psi},
\end{multline*}
that is,
$$
g_\fy e^{-i\scal k\cdo} = g_{\fy (\cdo +k)}.
$$

\par

A combination of these facts now gives
\begin{multline*}
c(g_\fy ,k) = c(g_\fy e^{-i\scal k{\cdo}},0)
\\[1ex]
=
c(g_{\fy (\cdo +k)},0) = \scal f{\fy (\cdo +k)}
= \scal {f(\cdo - k)}\fy ,
\end{multline*}
giving that
\begin{multline*}
\scal {Z_1f}{\fy \otimes \psi}
=
\sum _{k\in \zz d}\scal {f(\cdo -k)}\fy\scal {e^{i\scal k\cdo}}\psi
\\[1ex]
=
\sum _{k\in \zz d}c(g_\fy ,k) \scal {e^{i\scal k\cdo}}\psi
=\scal {g_\fy}\psi = \scal F{\fy \otimes \psi}.
\end{multline*}
Hence, if $F_0=F-Z_1f$, then $\scal {F_0}{\fy \otimes \psi}=0$
when $\fy \in \maclS _s^\sigma (\rr d)$ and $\psi \in \maclS _\sigma ^s(\rr d)$.
By Remark \ref{Rem:TensorUniqueness}
it now follows that $F=Z_1f$, which gives the result.
\end{proof}

\par

For completeness we also show that all quasi-periodic
distributions are tempered or Gelfand-Shilov distributions. (Cf.
\cite[Section 7.2]{Ho1}.) Here $(\maclD ^{\sigma ,s})'(\rr {2d})$
and $(\maclD ^{\sigma ,s;0})'(\rr {2d})$ are the duals of
$\maclE ^{\sigma ,s}(\rr {2d})\cap C^\infty _0(\rr {2d})$
and
$\maclE ^{\sigma ,s;0}(\rr {2d})\cap C^\infty _0(\rr {2d})$,
respectively.

\par

\begin{prop}\label{Prop:QPerGSDist}
Let $s,\sigma >1$ and $E$ be an ordered basis of
$\rr d$. Then the following is true:
\begin{enumerate}
\item The set of all quasi-periodic elements of
order $E$ in $\mascD '(\rr {2d})$ is equal to $\mascS '_{Z,E}(\rr {2d})$;

\vrum

\item The set of all quasi-periodic elements of
order $E$ in $(\maclD ^{\sigma ,s})'(\rr {2d})$ is equal to
$(\maclE ^{\sigma ,s}_{Z,E})'(\rr {2d})$;

\vrum

\item The set of all quasi-periodic elements of
order $E$ in $(\maclD ^{\sigma ,s;0}) '(\rr {2d})$ is equal to
$(\maclE ^{\sigma ,s;0}_{Z,E})'(\rr {2d})$.
\end{enumerate}
\end{prop}

\par

\begin{proof}
We only prove (2). The other assertions follow by
similar arguments and are left for the reader.

\par

Let $F\in C^\infty _{Z,E}(\rr {2d})$,
$\Phi \in \maclS _{s,\sigma}^{\sigma ,s}(\rr {2d})$
and let $\chi \in C_0^\infty (\rr {2d})
\cap \maclE ^{\sigma ,s}(\rr {2d})$ be such that
$$
\sum _{k\in \Lambda _E}\sum _{\kappa\in \Lambda _E'}
\chi (x+k,\xi +\kappa ) =1.
$$
If $\Phi \in \maclS _{s,\sigma}^{\sigma ,s}(\rr {2d})$,
then it follows by the quasi-periodisity of $F$
and some straight-forward computations that
\begin{align}
\scal F\Phi &= \scal F{T_\chi \Phi},
\label{Eq:QuasiPerAction1}
\intertext{where}
(T_\chi \Phi)(x,\xi )
&=
\sum _{k\in \Lambda _E}\sum _{\kappa \in \Lambda _E'}
e^{-i\scal k\xi}\Phi (x-k,\xi -\kappa )\chi (x,\xi ),
\label{Eq:QuasiPerAction2}
\end{align}
and that $T_\chi$ in \eqref{Eq:QuasiPerAction2}
is continuous from
$\maclS _{s,\sigma}^{\sigma ,s}(\rr {2d})$
to $C_0^\infty (\rr {2d})
\cap \maclE ^{\sigma ,s}(\rr {2d})$.

\par

By letting $\scal F\Phi$ be defined by the
right-hand side of \eqref{Eq:QuasiPerAction1}
when $F\in (\maclD ^{\sigma ,s})'(\rr {2d})$
and $\Phi \in \maclS _{s,\sigma}^{\sigma ,s}(\rr {2d})$,
it follows that $\Phi \mapsto \scal F{T_\chi \Phi}$
in \eqref{Eq:QuasiPerAction1}
defines a linear and continuous form on
$\maclS _{s,\sigma}^{\sigma ,s}(\rr {2d})$
which agree with the usual distribution action,
$\Phi \mapsto \scal F\Phi$ when
$\Phi \in C_0^\infty (\rr {2d})
\cap \maclE ^{\sigma ,s}(\rr {2d})$.
\end{proof}

\par

The mapping properties of Gelfand-Shilov distributions also lead
to some quieries concerning the inversion formula
\eqref{Eq:ZakTransfInv} for the Zak transform. Evidently, if
$F$ is a general quasi-periodic distribution or even Gelfand-Shilov
distribution, then the integral on the right-hand side of
\eqref{Eq:ZakTransfInv} might not be defined. On the other hand,
since $Z_E^{-1}F(x)$ is the zero order Fourier coefficient of the
expansion \eqref{Eq:DefZak}, it follows from \cite[Remark 2.3]{ToNa}
that the following is true. The details are left for the reader.

\par

\begin{prop}
Let $s,\sigma >0$ and $\phi \in \maclS _s^\sigma (\rr d)\setminus 0$
($\phi \in \Sigma _s^\sigma (\rr d)\setminus 0$) be fixed,
$f\in (\maclS _s^\sigma )'(\rr d)$ ($f\in (\Sigma _s^\sigma )'(\rr d)$),
and let $F=Z_Ef$. Then
\begin{multline}\tag*{(\ref{Eq:ZakTransfInv})$'$}
f(x) = (Z_E^{-1}F)(x)
\\[1ex]
= (\nm \phi{L^2}^2|Q_{E'}|)^{-1}
\int _{Q_{E'}}\left (
\int _{\rr d}(V_\phi F(x,\cdo  ))(\xi ,y)\widehat \phi (-y)e^{i\scal y\xi}\, dy
\right ) \, d\xi .
\end{multline}
\end{prop}

\par

\subsection{The Zak transform on Lebesgue and
modulation spaces}

\par

For completeness we begin the subsection by making a review of
the Zak transform when acting on Lebesgue spaces. Here we let
$$
\theta _1(p)=\min \left (0,1-\frac 2p \right )
\quad \text{and}\quad
\theta _2(p)=\max \left (0,1-\frac 2p \right ).
$$

\par

\begin{prop}\label{Prop:ZakOnLebesgue}
Let $E$ be an ordered basis of $\rr d$. Then the following is true:
\begin{enumerate}
\item if $p\in (0,2]$, then $Z_E$ from $\mascS (\rr d)$ to
$C^\infty (\rr {2d})$ is uniquely extendable to a continuous map
from $L^p(\rr d)$ to $L^p_{Z,E}(\rr {2d})$, and
$$
\nm {Z_Ef}{L^p_{Z,E}(\rr {2d})}\le \nm f{L^p(\rr d)},\qquad f\in L^p(\rr d) \text ;
$$

\vrum

\item if $p\in [2,\infty ]$, then $Z_E^{-1}$ from $C^\infty _{Z,E}(\rr {2d})$ to
$\mascS (\rr d)$ is uniquely extendable to a continuous map
from $L^p_{Z,E}(\rr {2d})$ to $L^p(\rr d)$, and
$$
\nm {f}{L^p(\rr d)}\le \nm {Z_Ef}{L^p_{Z,E}(\rr {2d})},
\qquad f\in L^p(\rr {d}).
$$

\vrum

\item if $f\in \mascS '(\rr d)$, $p\in [1,\infty ]$ and
$v\in \mascP _E(\rr d)$ satisfy $1/v\in L^1(\rr d)$, then
$$
\nm {f\cdot v^{\theta _1(p)}}{L^p(\rr d)}
\lesssim
\nm {Z_Ef}{L^p_{Z,E}(\rr {2d})}
\lesssim
\nm {f\cdot v^{\theta _2(p)}}{L^p(\rr d)}.
$$
\end{enumerate}
\end{prop}

\par

Proposition \ref{Prop:ZakOnLebesgue} (1) and (2) are evidently true for $p=2$,
in view of Proposition \ref{Prop:ZakTransfBasicMaps}, and
is presented in \cite[Lemma 3.1.2]{Ti}, without any proof in the case $p=1$.
In order to be self-contained, we give a proof in Appendix \ref{App:A}.

\par

When investigation mapping properties of the Zak
transform on modulation
spaces, we need to deduce various kinds of estimates
on short-time Fourier transforms and partial short-time
Fourier transforms of
Zak transforms. Especially we search suitable estimates on $V_\Phi (Z_E f)$,
and on
\begin{align}
(\mathsf{ZV}_{E ,\phi _1}^{(1)}f)(x,\xi ,\eta ) &= (V_{\phi _1}
(Z_E f(\cdo ,\xi )))(x ,\eta )
\label{Eq:STFTZak1}
\intertext{and}
(\mathsf{ZV}_{E ,\phi _2}^{(2)} f)(x,\xi ,y) &= (V_{\phi _2} (Z_E f(x ,\cdo )))(\xi ,y ),
\label{Eq:STFTZak2}
\end{align}
which are compositions of the Zak transform and the partial short-time
Fourier transforms with respect to the first and second variable, respectively.

\par

From the previous subsection it is clear that there is a one-to-one correspondence
between quasi-periodic functions and distributions, and Zak transforms of
functions and distributions. For a quasi-periodic function or distribution $F$
on $\rr {2d}$ which satisfies \eqref{Eq:QuasiPer}, and a suitable function or
distribution $\Phi$ on $\rr {2d}$, we have
\begin{equation}\label{Eq:STFTQuasiPer}
\begin{alignedat}{2}
(V_\Phi F)(x+ k,\xi ,\eta ,y) &= e^{-i\scal k\eta} 
(V_\Phi F)(x,\xi ,\eta ,y-k),
& \quad k &\in \Lambda _E,
\\[1ex]
(V_\Phi F)(x,\xi +\kappa ,\eta ,y) &=
e^{-i\scal y\kappa}(V_\Phi F)(x,\xi ,\eta ,y),
&\quad \kappa &\in \Lambda _E',
\end{alignedat}
\end{equation}
which follows by straight-forward computations.
We remark that functions
and distributions which satisfy conditions given
in \eqref{Eq:STFTQuasiPer}
are so-called \emph{echo-periodic functions and distributions},
considered in \cite{Toft22}.

\par

First we have the following result concerning
identifying Lebesgue spaces
via estimates of corresponding Zak transforms.

\par

\begin{thm}\label{Thm:ZakLebesgue}
Let $E$ be an ordered basis of $\rr d$,
$p,r\in (0,\infty ]$, $\omega \in
\mascP _E(\rr d)$, $\phi \in \Sigma _1(\rr d)
\setminus 0$,
and let $f$ be a Gelfand-Shilov distribution on
$\rr d$. Then
\begin{align}
\nm f{L^p_{(\omega )}}
&\asymp
\nm
{G_{E,r,\omega ,f}}{L^{p}(Q_E\times \rr d)},
\label{Eq:ZakLebesgue}
\intertext{where}
G_{E,r,\omega ,f}(x,y)
&\equiv
\nm{(\mathsf{ZV}_{E,\phi}^{(2)} f)(x,\cdo ,y)\omega (-y)}{L^r(Q_{E'})}.
\intertext{In particular,}
\nm f{L^p}
&\asymp
\nm{\mathsf{ZV}_{E ,\phi}^{(2)} f}
{L^p(Q _{E\times E'}\times \rr d)}.
\end{align}
\end{thm}

\par

\begin{proof}
We only prove the result for $p<\infty$. The case $p=\infty$ follows
by similar arguments and is left for the reader.

\par

The distribution $\xi \mapsto Z_Ef(x,\xi )$ is
$E'$-periodic, and it follows from
\eqref{Eq:ModPerNormEquiv3} that
\begin{multline*}
\left ( \sum _{j\in \Lambda _E} |f(x-j)
\omega (x-j)|^p \right )^{\frac 1p}
\asymp
\left ( \sum _{j\in \Lambda _E}
|f(x-j)\omega (-j)|^p \right )^{\frac 1p}
\\[1ex]
\asymp \nm {G_{E,r,\omega ,f}(x,\cdo )}{L^p(\rr d)},
\quad x\in Q_E.
\end{multline*}
The result now follows by apply the
$L^p(Q_E)$ quasi-norm with respect to
the $x$-variable.
\end{proof}

\par

In the same way we may identify modulation spaces by 
using the Zak transform as in the next result.
We also recall Definition \ref{Def:ExtClassModSpaces}
and Remark \ref{Rem:WienerSpace} for definitions
and notions concerning the Wiener amalgam space
$W^{\mabfp}_{E,(\omega )}(\rr d)$.

\par

\begin{thm}\label{Thm:ZakModulation}
Let $E,E_0$ be an ordered bases of $\rr d$,
$\mabfp ,\mabfr \in (0,\infty ]^{2d}$,
$\omega _0\in \mascP _E(\rr {2d})$ and
$\omega \in \mascP _E(\rr {4d})$
be such that \eqref{Eq:ZakModulationWeight} holds.
Then $Z_{E}$ from $\Sigma _1(\rr d)$ to
$C^\infty (\rr {2d})$ is uniquely
extendable to a homeomorphism from
$M^{\mabfp}_{E\times E_0',(\omega _0)}(\rr d)$
to $W^{\mabfr ,\mabfp}_{Z,E,E_0,(\omega )}(\rr {2d})$,
and
\begin{equation}\label{Eq:ZakModulation}
\nm {f}{M^{\mabfp}_{E\times E_0',(\omega _0)}}
\asymp
\nm {Z_{E} f}{W^{\mabfr ,\mabfp}_{Z,E,E_0,(\omega )}},
\qquad f\in \Sigma _1'(\rr d).
\end{equation}
\end{thm}

\par

\begin{proof}
First we prove the result for $\mabfr =\infty$.
Let $\Phi =\phi _1\otimes \phi _2$ with
$\phi _1,\phi _2\in \Sigma _1(\rr d)\setminus 0$.
By straight-forward computations we get
\begin{equation}\label{Eq:PerPartZakSTFT}
(\mathsf{ZV}_{1,\phi _1}^{(1)}f)(x,\xi ,\eta )
=
\sum _{j\in \Lambda _E}\big ( (V_{\phi _1}f)(x-j,\eta )
e^{-i\scal j\eta}\big )e^{i\scal j\xi }.
\end{equation}
Let
$$
\mabfp _1=(p_1,\dots ,p_d)
\quad \text{and}\quad
\mabfp _2=(p_{d+1},\dots ,p_{2d})
$$
when
$$
\mabfp = (p_1,\dots ,p_{2d}),
$$
and consider the functions
\begin{align*}
F(x,\eta )
&=
\nm {(V_{\phi _1}f)(x-\cdo ,\eta )
\omega _0(x-\cdo ,\eta )}
{\ell ^{\mabfp _1}_{E}(\Lambda _{E})}.
\\[1ex]
g_0 (x)&= \nm {F(x,\cdo )}{L^{\mabfp _2}_{E_0'}(\rr d)}
\\[1ex]
G_0(x,\xi ,\eta ,y)
&\equiv
|V_\Phi(Z_E f)(x,\xi ,\eta ,y)\omega _0(x-y,\eta )|,
\\[1ex]
G(x,\xi ,\eta )
&\equiv
\nm {G_0(x,\xi ,\eta ,\cdo )}
{L^{\mabfp _1}_{E}(\rr d)},
\\[1ex]
H(x,\xi )
&=
\nm {G(x,\xi ,\cdo )}{L^{\mabfp _2}_{E_0'}(\rr d)},
\intertext{and}
h_0(x) &= h_{0,r_0}(x) = \nm {H(x,\cdo )}{L^{r_0}_{E_0'}(Q _{E_0'})},
\quad r_0\in (0,\infty ].
\end{align*}
Since $\xi
\mapsto
(\mathsf{ZV}_{1,\phi _1}^{(1)}f)(x,\xi ,\eta )$
is $E'$-periodic
with Fourier coefficients
$$
j\mapsto (V_{\phi _1}f)(x-j,\eta )
e^{-i\scal j\eta}
$$
(cf. \eqref{Eq:PerPartZakSTFT}),
and the (partial) short-time Fourier transform of that distribution equals
$V_\Phi (Z_1 f)$, it
follows from \eqref{Eq:ModPerNormEquiv3} that
\begin{equation}\label{Eq:PerCoeffEquivHere1}
F(x,\eta ) \asymp \nm {G(x,\cdo ,\eta )}{L^{r}_{E'}(Q_{E'})},
\quad r\in (0,\infty ].
\end{equation}

\par

First let $r_0\le \min (\mabfp )$. If we apply the 
$L^{\mabfp _2}_{E_0'}$ norm on
\eqref{Eq:PerCoeffEquivHere1} with respect to the
$\eta$ variable
and using H{\"o}lder's inequality we get
\begin{multline}\label{Eq:f0h0Est1}
g_0(x)
=
\nm {F(x,\cdo )}{L^{\mabfp _2}_{E_0'}(\rr d)}
\asymp
\nm {G(x,\cdo )}{L^{r_0,\mabfp _2}_{E'\times E_0'}(Q_{E'} \times \rr d)}
\\[1ex]
\lesssim
\nm {H(x,\cdo )}{L^{r_0}_{E'}(Q_{E'})} = h_{0,r_0}(x).
\end{multline}
If
$$
g_1(x) = \nm {F_1(x,\cdo )}{\ell ^{\mabfp _2}_{E_0'}}
\quad \text{with}\quad
F_1(x,\iota ) = \nm {F(x,\iota )}{L^{r_0}_{E_0'}(\iota +Q _{E_0'})},
$$
then the fact that $r_0\le \min (\mabfp )$ and Jensen's inequality give
$g_1 \lesssim g_0$.

\par

By applying the $L^{r_0}_{E}(Q_E)$ norm on the latter inequality,
using the fact that
$$
\omega _0(x-y,\eta )
\asymp \omega _0(-y,\eta ),
\qquad x\in Q_E
$$
and Jensen's inequality again we obtain
$$
\nm {a_{1,f}}{\ell ^{\mabfp}_{E_1}}\lesssim \nm {g_0}{L^{r_0}(Q_E)},
\quad \text{where}\quad
a_{1,f}(j,\iota )
=
\nm {V_{\phi _1}f\cdot \omega _0}{L^{r_0}_{E_1}((j,\iota )+Q _{E_1})},
$$
where $E_1=E\times E_0'$. That is,
$$
\nm {V_{\phi _1}f}{\sfW ^{r_0}_{E_1}
(\omega _0,\ell ^{\mabfp}_{E_1}
(\Lambda _{E_1}))}
\lesssim
\nm {g_0}{L^{r_0}(Q_E)},
$$
which is the same as
\begin{equation}\label{Eq:f0fEst1}
\nm f{M^{\mabfp}_{E\times E_0',(\omega _0)}}
\lesssim
\nm {g_0}{L^{r_0}(Q_E)}
\end{equation}
in view of Proposition \ref{Prop:WienerEquiv}.

\par

In order to estimate $h_{0,r_0}$ we apply 
\eqref{Eq:STFTQuasiPer}
to get
\begin{multline*}
|V_\Phi (Z_{E} f)(x+j,\xi +\iota ,\eta ,y)
\omega _0(x+j-y,\eta )|
\\[1ex]
=
|V_\Phi (Z_{E}f)(x,\xi ,\eta ,y-j)
\omega _0(x-y+ j,\eta )|,\quad
(j,\iota )\in \Lambda _{E\times E'}.
\end{multline*}
By first applying the $L^{\mabfp_1}_{E}(\rr {d})$
norm with respect to the $y$ variable and
then the $L^{\mabfp_2}_{E_0'}(\rr {d})$
norm with respect to the $\eta$ variable we get
$$
H(x+j,\xi +\iota )= H(x,\xi ),\qquad (j,\iota )\in \Lambda _{E\times E'}.
$$

\par

Hence, by applying the $L^{r_0}_{E}(Q_E)$ norm
on $h_{0,r_0}$ and using
H{\"o}lder's and Jensen's inequalities we get
\begin{multline}\label{Eq:h0Est1}
\nm {h_{0,r_0}}{L^{r_0}_{E}(Q_E)}
= \nm H{L^{r_0}_{E\times E'}(Q _{E\times E'})}
\lesssim
\nm H{L^\infty_{E\times E'}(Q _{E\times E'})}
\\[1ex]
=
\sup _{(j,\iota )\in \Lambda _{E\times E'}}
\left (
\nm H{L^\infty_{E\times E'}((j,\iota )+Q _{E\times E'})}
\right ),
\end{multline}
A combination of
\eqref{Eq:f0h0Est1}--\eqref{Eq:h0Est1}
and Proposition \ref{Prop:WienerEquiv}
now gives
\begin{equation}\label{Eq:ZakModulation2}
\nm {f}{M^{\mabfp}_{E\times E_0',(\omega _0)}}
\lesssim
\nm {Z_{E} f}{W^{\infty ,\mabfp}_{Z,E,E_0,(\omega )}},
\qquad f\in \Sigma _1'(\rr d).
\end{equation}

\par

In order to get the reversed estimate we again
apply the $L^{\mabfp _2}_{E_0'}$ norm on
\eqref{Eq:PerCoeffEquivHere1} with respect to
the $\eta$ variable
and use H{\"o}lder's inequality to get
\begin{multline}\label{Eq:f0h0Est2}
g_0(x) = \nm {F(x,\cdo )}{L^{\mabfp _2}_{E_0'}(\rr d)}
\asymp
\nm {G(x,\cdo )}
{L^{\infty ,\mabfp _2}
_{E'\times E_0'}(Q_{E'}
\times \rr d)}
\\[1ex]
\gtrsim
\nm {H(x,\cdo )}
{L^{\infty}_{E'}(Q_{E'})}
=
h_{0,\infty}(x).
\end{multline}
If
$$
g_2(x) = \nm {F_2(x,\cdo )}{\ell ^{\mabfp _2}_{E_0'}}
\quad \text{with}\quad
F_2(x,\iota ) = \nm {F(x,\cdo )}
{L^{\infty}_{E_0'}(\iota +Q _{E_0'})},
$$
then Jensen's inequality give
$g_0 \lesssim g_2$.

\par

By applying the $L^{\infty}_{E}
(Q_E)$ norm on the latter inequality
and using Jensen's inequality again we obtain
$$
\nm {a_{2,f}}{\ell ^{\mabfp}_{E_1}}\gtrsim \nm {g_0}{L^{\infty}(Q_E)},
\quad \text{where}\quad
a_{2,f}(j,\iota )
=
\nm {V_{\phi _1}f\cdot \omega _0}{L^{\infty}((j,\iota )+Q _{E_1})}.
$$
That is,
$$
\nm {V_{\phi _1}f}
{\sfW ^{\infty}_{E_1}(\omega _0,
\ell ^{\mabfp}_{E_1}(\Lambda _{E_1}))}
\gtrsim
\nm {g_0}{L^{r_0}(Q_E)},
$$
which is the same as
\begin{equation}\label{Eq:f0fEst2}
\nm f{M^{\mabfp}_{E\times E_0',(\omega _0)}}
\gtrsim
\nm {g_0}{L^{\infty}(Q_E)}
\end{equation}
in view of Proposition \ref{Prop:WienerEquiv}.

\par

By applying the
$L^{\infty}_{E}(Q_E)$ norm on
$h_{0,\infty}$ and using \eqref{Eq:h0Est1}
we get
\begin{equation}\label{Eq:h0Est2}
\nm {h_{0,\infty}}{L^{\infty}_{E}(Q_E)}
=
\nm H{L^{\infty}_{E\times E'}(Q _{E\times E'})}
\asymp
\nm {Z_{E} f}{W^{\infty ,\mabfp}_{Z,E,E_0,(\omega )}},
\end{equation}
where the last relation follows
from Proposition \ref{Prop:WienerEquiv}.
A combination of \eqref{Eq:f0h0Est2}, \eqref{Eq:f0fEst2}
and \eqref{Eq:h0Est2} now gives
\begin{equation}\label{Eq:ZakModulation3}
\nm {f}{M^{\mabfp}_{E\times E_0',(\omega _0)}}
\lesssim
\nm {Z_{E} f}{W^{\infty ,\mabfp}_{Z,E,E_0,(\omega )}},
\qquad f\in \Sigma _1'(\rr d),
\end{equation}
and the result in the case $\mabfr =\infty$ follows by combining
\eqref{Eq:ZakModulation3} with
\eqref{Eq:ZakModulation2}.

\par

For general $\mabfr \in (0,\infty ]^{2d}$, we notice that the echo-periodicity
\eqref{Eq:STFTQuasiPer} implies that $H_{2,F,\omega ,E,E_0,\mabfp}$ in
\eqref{Eq:STFTZakPartNorm2} is $E\times E'$ periodic.
The general case now follows from Proposition
\ref{Prop:WienerEquiv}, Theorem  \ref{Thm:ZakModulation}
and the previous observation. The details are left for the reader.
\end{proof}

\par

A consequence of the previous result is that
$W_{Z,E,E_0,(\omega )}^{\mabfr , \mabfp}(\rr {2d})$ is independent
of $\mabfr$ (also in topological sense), and for this reason we set
$$
W_{Z,E,E_0,(\omega )}^{\mabfp}
=
W_{Z,E,E_0,(\omega )}^{\mabfr , \mabfp}.
$$

\par

As a special case of the previous result we have the following.

\par

\begin{cor}\label{Cor:ZakModulation}
Let $E$ be an ordered basis in $\rr d$, $\Phi \in \Sigma _1(\rr {2d})
\setminus 0$ and let $p\in (0,\infty ]$. Then
\begin{align*}
\nm {f}{M^{p}}
&\asymp
\nm{(V_\Phi (Z_{E} f))}{L^p(Q _{E\times E'}
\times \rr {2d})},
\qquad f\in \Sigma _1'(\rr d).
\end{align*}
\end{cor}

\par

\section{Duality properties and some further characterizations
of quasi-peridic elements}\label{sec3}

\par

In this section we discuss various aspects concerning duality and characterizations
for quasi-periodic elements, as well as transitions of linear operators under the Zak
transform.  In Subsection \ref{subsec3.1} we show that the elements in
$\maclE ^{\sigma ,s}_{Z,E}$ and $(\maclE ^{\sigma ,s}_{Z,E})'$ can be
completely characterized by estimates on their short-time Fourier transform.
Thereafter we use these characterizations in Subsection \ref{subsec3.2} to show
that the $L^2_{Z,E}$ form on $\maclE ^{\sigma ,s}_{Z,E}$
is uniquely extendable to a continuous sesqui-linear form on
$(\maclE ^{\sigma ,s}_{Z,E})' \times \maclE ^{\sigma ,s}_{Z,E}$,
and that the dual of $\maclE ^{\sigma ,s}_{Z,E}$ can be identified
by $(\maclE ^{\sigma ,s}_{Z,E})'$ through this form. We conclude
the section by showing in which ways linear operators are transformed by the
Zak transform (cf. Subsection \ref{subsec3.3}).

\par

\subsection{Characterizations of quasi-periodic elements via estimates on
their short-time Fourier transform}\label{subsec3.1}

\par

The following results are analogous to Propositions 2.7 and 2.8
in \cite{ToNa} concerning characterizations of periodic elements in Gelfand-Shilov
distribution spaces, and to Proposition \ref{stftGelfand2}. Here and in what
follows we set
\begin{equation}\label{Eq:OmegaEDef}
\Omega _E=Q _{E\times E'}\times \rr {2d}\subseteq \rr {4d}
\end{equation}
when $E$ is an ordered basis of $\rr d$.

\par

\begin{prop}\label{Prop:QuasiPerNonDistSTFTChar}
Let $s,\sigma >0$, $E$ be an ordered basis of $\rr d$,
$\Omega _E$ be as in \eqref{Eq:OmegaEDef},
$F$ be a quasi-periodic
Gelfand-Shilov distribution with respect to $E$ and let
$\phi \in \maclS _{s,\sigma}^{\sigma ,s}(\rr {2d})\setminus 0$
($\phi \in \Sigma _{s,\sigma}^{\sigma ,s}(\rr {2d})\setminus 0$).
Then the following conditions are equivalent:
\begin{enumerate}
\item[(1)] $F\in \maclE ^{\sigma ,s}_{Z,E}(\rr {2d})$ ($F\in
\maclE ^{\sigma ,s;0}_{Z,E}(\rr {2d})$);

\vrum

\item[(2)] for some $r>0$ (for every $r>0$), it holds
\begin{equation}\label{Eq:QuasiPerDistSTFTChar1}
|(V_\phi F)(x,\xi ,\eta ,y)|
\lesssim
e^{-r(|\eta |^{\frac 1\sigma}+|y|^{\frac 1s})},
\quad 
x\in Q_E,\ \xi ,\eta ,y \in \rr {d}\text ;
\end{equation}

\vrum

\item[(3)] for some $r>0$ (for every $r>0$), it holds
\begin{equation}\label{Eq:QuasiPerDistSTFTChar2}
|(V_\phi F)(x,\xi ,\eta ,y)|
\lesssim
e^{-r(|\eta |^{\frac 1\sigma}+|x-y|^{\frac 1s})},
\quad 
x,\xi ,\eta ,y \in \rr {d}.
\end{equation}
\end{enumerate}
\end{prop}

\par

\begin{prop}\label{Prop:QuasiPerDistSTFTChar}
Let $s,\sigma >0$, $E$ be an ordered basis of $\rr d$,
$\Omega _E$ be as in \eqref{Eq:OmegaEDef},
$F$ be a quasi-periodic
Gelfand-Shilov distribution with respect to $E$ and let
$\phi \in \maclS _{s,\sigma}^{\sigma ,s}(\rr {2d})\setminus 0$
($\phi \in \Sigma _{s,\sigma}^{\sigma ,s}(\rr {2d})\setminus 0$).
Then the following conditions are equivalent:
\begin{enumerate}
\item[(1)] $F\in (\maclE ^{\sigma ,s}_{Z,E})'(\rr {2d})$ ($F\in
(\maclE ^{\sigma ,s;0}_{Z,E})'(\rr {2d})$);

\vrum

\item[(2)] for every $r>0$ (for some $r>0$), it holds
\begin{equation}\label{Eq:QuasiPerDistSTFTChar3}
|(V_\phi F)(x,\xi ,\eta ,y)|
\lesssim
e^{r(|\eta |^{\frac 1\sigma}+|y|^{\frac 1s})},
\quad
x\in Q_E,\ \xi ,\eta ,y \in \rr {d}\text ;
\end{equation}

\vrum

\item[(3)] for every $r>0$ (for some $r>0$), it holds
\begin{equation}\label{Eq:QuasiPerDistSTFTChar4}
|(V_\phi F)(x,\xi ,\eta ,y)|
\lesssim
e^{r(|\eta |^{\frac 1\sigma}+|x-y|^{\frac 1s})},
\quad 
x,\xi ,\eta ,y \in \rr {d}.
\end{equation}
\end{enumerate}
\end{prop}

\par

We only prove Proposition \ref{Prop:QuasiPerDistSTFTChar}, and then only in
the Roumieu case. The Beurling case of Proposition \ref{Prop:QuasiPerDistSTFTChar}
as well as Proposition \ref{Prop:QuasiPerNonDistSTFTChar} follow by similar
arguments and are left for the reader.

\par

\begin{proof}
Since it is obvious that (3) implies (2), the result follows if we prove that (2) implies (1) and
(1) implies (3).

\par

Suppose \eqref{Eq:QuasiPerDistSTFTChar3} holds for every $r>0$, and let
$k=k_x\in \Lambda _E$ and $\kappa =\kappa _\xi \in \Lambda _E'$ be
chosen such that $x_0=x-k\in Q_E$ and $\xi -\kappa \in Q_{E'}$, when
$x,\xi \in \rr d$ are given. Then \eqref{Eq:STFTQuasiPer}
and \eqref{Eq:QuasiPerDistSTFTChar3} give
\begin{multline*}
|V_\phi F(x,\xi ,\eta ,y)| = |V_\phi F(x_0+k,\xi _0+\kappa ,\eta ,y)|
=
|V_\phi F(x_0,\xi _0,\eta ,y-k)|
\\[1ex]
\lesssim
e^{r(|\eta |^{\frac 1\sigma}+|y-k|^{\frac 1s})}
\lesssim
e^{Cr(|k|^{\frac 1s} +|\eta |^{\frac 1\sigma}+|y|^{\frac 1s})}
\\[1ex]
\asymp
e^{Cr(|x|^{\frac 1s} +|\eta |^{\frac 1\sigma}+|y|^{\frac 1s})}
\le
e^{Cr(|x|^{\frac 1s} +|\xi |^{\frac 1\sigma}+|\eta |^{\frac 1\sigma}+|y|^{\frac 1s})},
\end{multline*}
for every $r>0$. By Proposition Proposition \ref{stftGelfand2}, it now follows that
$F\in (\maclE ^{\sigma ,s}_{Z,E})'(\rr {2d})$, and we have proved that (2)
implies (1).

\par

It remains to prove that (1) implies (3). Suppose that (1) holds, let $x,\xi ,\eta ,y\in \rr d$
be arbitrary, and choose $x_0\in Q_E$, $\xi _0\in Q_{E'}$, $k\in \Lambda _E$
and $\kappa \in \Lambda _E'$ such that
$$
x=x_0+k
\quad \text{and}\quad
\xi = \xi _0+\kappa .
$$
Then (1) and Proposition \ref{stftGelfand2} give
\begin{multline*}
|(V_\phi F)(x,\xi ,\eta ,y)|
=
|(V_\phi F)(x_0+k,\xi _0+\kappa ,\eta ,y)|
=
|(V_\phi F)(x_0,\xi _0,\eta ,y-k)|
\\[1ex]
\lesssim
e^{r(|x_0|^{\frac 1s}+|\xi _0|^{\frac 1\sigma} +|\eta |^{\frac 1\sigma}+|y-k|^{\frac 1s})}
\asymp
e^{r(|\eta |^{\frac 1\sigma}+|x-y|^{\frac 1s})},
\end{multline*}
and we have proved that (1) implies (3).
\end{proof}

\par

\subsection{Duality properties of Gevrey type
quasi-period elements}\label{subsec3.2}

\par

We shall next use the previous characterizations in
Propositions \ref{Prop:QuasiPerNonDistSTFTChar} and
\ref{Prop:QuasiPerDistSTFTChar} to show that the form
\eqref{Eq:QuasiPerScalarProd} can be written as
\begin{equation}\label{Eq:QuasiPerForm}
(F,G)_{Z,E}\equiv (|Q_{E'}|\nm \phi{L^2}^2)^{-1}
\cdot (V_\phi F ,V_\phi G)_{L^2(\Omega _E)},
\end{equation}
when $F,G\in L^2_{Z,E}(\rr {2d})$.
Here $\phi \in \maclS _{s,\sigma}^{\sigma ,s}(\rr {2d})\setminus 0$
is fixed and $\Omega _E$ is given by \eqref{Eq:OmegaEDef}.
We use this identity to extend the definition of this form to permit
$$
F\in \maclE ^{\sigma ,s}_{Z,E}(\rr {2d})
\quad \text{and}\quad
G\in (\maclE ^{\sigma ,s}_{Z,E})'(\rr {2d}).
$$
We also show that the dual of $\maclE ^{\sigma ,s}_{Z,E}(\rr {2d})$
is equal to $(\maclE ^{\sigma ,s}_{Z,E})'(\rr {2d})$ through this form,
and similarly when $\maclE ^{\sigma ,s}_{Z,E}$ are replaced by
$\maclE ^{\sigma ,s;0}_{Z,E}$ at each occurrence.

\par

\begin{rem}\label{Rem:FormSense}
By Propositions \ref{Prop:QuasiPerNonDistSTFTChar} and
\ref{Prop:QuasiPerDistSTFTChar} it follows that for $V_\phi F$ and $V_\phi G$ in
\eqref{Eq:QuasiPerForm} we have
$$
V_\phi F (x,\xi ,\eta ,y)\overline{V_\phi G(x,\xi ,\eta ,y)}
\in L^1(\Omega _E).
$$
Hence the right-hand side of
\begin{multline}\tag*{(\ref{Eq:QuasiPerForm})$'$}
(F,G)_{Z,E}
\\[1ex]
=
(|Q_{E'}|\cdot \nm \phi{L^2}^2)^{-1}
\iiiint _{\Omega _E}
V_\phi F (x,\xi ,\eta ,y)\overline{V_\phi G(x,\xi ,\eta ,y)}\, dxd\xi d\eta dy
\end{multline}
makes sense and we may evaluate the integrals with respect to $x,\xi ,\eta ,y$ in any order.
It also follows that the map $(F,G)\mapsto (F,G)_{Z,E}$
defines a continuous map from $\maclE ^{\sigma ,s}_{Z,E}(\rr {2d})\times
(\maclE ^{\sigma ,s}_{Z,E})'(\rr {2d})$ to $\mathbf C$.
\end{rem}

\par

%

\begin{thm}\label{Thm:DualQuasiPer}
Let $(\cdo ,\cdo )_{Z,E}$ be given by \eqref{Eq:QuasiPerForm}.
Then the following is true:
\begin{enumerate}
\item $(\cdo ,\cdo )_{Z,E}$ on $\maclE ^{\sigma ,s}_{Z,E}(\rr {2d})\times
\maclE ^{\sigma ,s}_{Z,E}(\rr {2d})$ is uniquely extendable to a continuous sesqui-linear
map from $L^2_{Z,E}(\rr {2d})\times L^2_{Z,E}(\rr {2d})$ and from
$(\maclE ^{\sigma ,s}_{Z,E})'(\rr {2d})\times \maclE ^{\sigma ,s}_{Z,E}(\rr {2d})$ to
$\mathbf C$, and
\begin{equation}\label{Eq:ScalProdLink}
(F,G)_{Z,E} = |Q_{E'}|^{-1} (F,G)_{L^2(Q _{E\times E'})},
\qquad F,G\in L^2_{Z,E}(\rr {2d})\text ;
\end{equation}

\vrum

\item if $f\in (\maclS _s^\sigma )'(\rr d)$, $f\in \maclS _s^\sigma (\rr d)$, $F=Z_Ef$
and $G=Z_Eg$, then $(F,G)_{Z,E} = (f,g)_{L^2(\rr d)}$;

\vrum

\item the dual of $\maclE ^{\sigma ,s}_{Z,E}(\rr {2d})$ is equal to
$(\maclE ^{\sigma ,s}_{Z,E})'(\rr {2d})$ through the form $(\cdo ,\cdo )_{Z,E}$;
\end{enumerate}
The same holds true with $\maclE ^{\sigma ,s;0}_{Z,E}$
and $\Sigma _s^\sigma$ in place of $\maclE ^{\sigma ,s}_{Z,E}$
respective $\maclS _s^\sigma$ at each occurrence.
\end{thm}

\par

\begin{proof}
By Proposition \ref{Prop:ZakTransfBasicMaps},
Theorems \ref{Thm:ZakTestFunctions}, \ref{Thm:ZakDist}, and the facts that
$\maclS _s^\sigma (\rr d)$ is dense in $L^2(\rr d)$ and
the dual of $\maclS _s^\sigma (\rr d)$ equals $(\maclS _s^\sigma )'(\rr d)$ through
the form $(\cdo ,\cdo )_{L^2(\rr d)}$, it suffices to prove (2).

\par

Let
\begin{gather*}
f\in (\maclS _{s}^\sigma )'(\rr d),\quad g\in \maclS _{s}^\sigma (\rr d),\quad
\phi \in \maclS _{s,\sigma}^{\sigma ,s}(\rr {2d})\setminus 0,
\\[1ex]
\psi (x,y) = \mascF ^{-1}(\phi (x,\cdo ))(y),
\quad \text{and}\quad \psi _y(x)=\psi (x,y).
\end{gather*}
Then it follows by straight-forward computations that
\begin{alignat*}{2}
V_\phi (Z_Ef)(x,\xi ,\eta ,y)
&=
e^{-i\scal y\xi} F_{X}(\xi ),
& \qquad
F_{X}(\xi ) &=
\sum _{j\in \Lambda _E}
c_{X}(f,j)e^{i\scal {j}\xi}
\intertext{and}
V_\phi (Z_Eg)(x,\xi ,\eta ,y)
&=
e^{-i\scal y\xi} G_{X}(\xi ),
& \qquad
G_{X}(\xi ) &=
\sum _{j\in \Lambda _E}
c_{X}(g,j)e^{i\scal {j}\xi}
\end{alignat*}
where $X=(x,\eta ,y)\in Q_E\times \rr {2d}$,
and the Fourier coefficients $c_{X}(f,j)$ and $c_{X}(g,j)$ are given by
\begin{align}
c_{X}(f,j) &= (V_{\psi _{y-j}}f)(x-j,\eta )e^{-i\scal j\eta}
\label{Eq:DualFourCoeff1}
\intertext{and}
c_{X}(g,j) &= (V_{\psi _{y-j}}g)(x-j,\eta )e^{-i\scal j\eta}.
\label{Eq:DualFourCoeff2}
\end{align}
Since the short-time Fourier transforms $V_\phi (Z_Ef)$ and $V_\phi (Z_Eg)$
are smooth, it follows that $\xi \mapsto F_{X}(\xi )$ and $\xi \mapsto G_{X}(\xi )$
are smooth periodic functions for
every $X$. Hence the Fourier coefficients in \eqref{Eq:DualFourCoeff1}
and \eqref{Eq:DualFourCoeff2} satisfies
\begin{equation}\label{Eq:DualFourCoeffEst}
|c_{X}(f,j)| \lesssim \eabs j^{-N}
\quad \text{and}\quad
|c_{X}(g,j)| \lesssim \eabs j^{-N}
\end{equation}
for every $N\ge 0$, when $X\in Q_E\times \rr {2d}$ is fixed.
By integrating 
\begin{multline*}
V_\phi (Z_Ef)(x,\xi ,\eta ,y)\overline{V_\phi (Z_Eg)(x,\xi ,\eta ,y)}
\\[1ex]
=
\left (
\sum _{j\in \Lambda _E}
c_{X}(f,j)e^{i\scal {j}\xi}
\right )
\overline{
\left (
\sum _{j\in \Lambda _E}
c_{X}(g,j)e^{i\scal {j}\xi}
\right )
},
\end{multline*}
with respect to the $\xi$ variable and using \eqref{Eq:DualFourCoeffEst}, we obtain
\begin{multline*}
\int _{Q_{E'}}
V_\phi (Z_Ef)(x,\xi ,\eta ,y)\cdot 
\overline{V_\phi (Z_Eg)(x,\xi ,\eta ,y)}\, d\xi
\\[1ex]
=
\int _{Q_{E'}}
\left (
\sum _{j\in \Lambda _E}
c_{X}(f,j) e^{i\scal {j}\xi}
\right )
\overline {
\left (
\sum _{j\in \Lambda _E}
c_{X}(g,j) e^{i\scal {j}\xi}
\right )
}
\, d\xi
\\[1ex]
=
|Q_{E'}|
\sum _{j\in \Lambda _E}
c_{X}(f,j)\overline {c_{X}(g,j)}
\\[1ex]
=
|Q_{E'}|
\sum _{j\in \Lambda _E}
(V_{\psi _{y-j}}f)(x-j,\eta )\overline{(V_{\psi _{y-j}}g)(x-j,\eta )}.
\end{multline*}
Hence, integrating with respect to $x,\eta ,y$, using Moyal's
identity and Remark \ref{Rem:FormSense}, we obtain
\begin{multline}\label{Eq:LastFormEvaluation}
\nm \phi{L^2(\rr d)}^2\cdot (F,G)_{Z,E}
\\[1ex]
=
\iint _{Q_E\times \rr d}
\left (
\sum _{j\in \Lambda _E}
\int _{\rr d} 
(V_{\psi _{y-j}}f)(x-j,\eta )\overline{(V_{\psi _{y-j}}g)(x-j,\eta )}
\, dy
\right )
\, dxd\eta
\\[1ex]
=
\iint _{Q_E\times \rr d}
\left (
\sum _{j\in \Lambda _E}
\int _{\rr d} 
(V_{\psi _{y}}f)(x-j,\eta )\overline{(V_{\psi _{y}}g)(x-j,\eta )}
\, dy
\right )
\, dxd\eta
\\[1ex]
=
\int _{\rr d} 
\left (
\iint _{\rr {2d}}
(V_{\psi _{y}}f)(x,\eta )\overline{(V_{\psi _{y}}g)(x,\eta )}
\, dxd\eta
\right )
\, dy
\\[1ex]
=
\int _{\rr d} 
\left (
(f,g)_{L^2(\rr d)}\nm {\psi _y}{L^2(\rr d)}^2
\right )
\, dy
= \nm \phi{L^2(\rr {2d})}^2(f,g)_{L^2(\rr d)},
\end{multline}
which gives (2). Here we observe that the estimates in Proposition \ref{stftGelfand2}
implies that the involved expressions in \eqref{Eq:LastFormEvaluation}
possess suitable $L^1$ properties, which allow us to swap the orders of summations
and integrations. This gives the result.
\end{proof}

\par

\subsection{Duality properties of Banach spaces of
quasi-periodic elements}\label{subsec3.3}

\par

In the following we use the links Theorem \ref{Thm:ZakModulation}
and \eqref{Eq:ScalProdLink} to carry over duality
properties of Lebesgue and modulation spaces to quasi-periodic
elements in Lebesgue and Wiener type spaces. Here $p'\in [1,\infty ]$
denotes the conjugate exponent to $p\in [1,\infty ]$, i.{\,}e. $p$
and $p'$ should satisfy $\frac 1p+\frac 1{p'}=1$. Furthermore,
we let $\mabfp '=(p_1',\dots ,p_d')$ when $\mabfp =(p_1,\dots ,p_d)$.

\par

\begin{thm}\label{Thm:DualQuasiPer2}
Let  $\mabfp \in [1,\infty ]^{2d}$, $E$ and $E_0$
be ordered bases of $\rr d$, $\omega \in \mascP _E(\rr {4d})$
and $\omega _0\in \mascP _E(\rr {2d})$ be such that
\eqref{Eq:ZakModulationWeight} holds. Then the following is true:
\begin{enumerate}
\item the map $(F,G)\mapsto (F,G)_{Z,E}$ from $\maclE _{Z,E}^{1,1;0}(\rr {2d})
\times \maclE _{Z,E}^{1,1;0}(\rr {2d})$ to $\mathbf C$ is uniquely extendable
to a continuous mapping from $W^{\mabfp}_{Z,E,E_0,(\omega )}(\rr {2d})
\times W^{\mabfp '}_{Z,E,E_0,(1/\omega )}(\rr {2d})$
to $\mathbf C$. If in addition $\max (\mabfp )<\infty$, then the dual of
$W^{\mabfp}_{Z,E,E_0,(\omega )}(\rr {2d})$ can be identified
with $W^{\mabfp '}_{Z,E,E_0,(1/\omega )}(\rr {2d})$ through the form
$(\cdo ,\cdo )_{Z,E}$;

\vrum

\item if $f\in M^{\mabfp}_{E\times E_0',(\omega _0)}(\rr d)$,
$g\in M^{\mabfp '}_{E\times E_0',(1/\omega _0)}(\rr d)$, then
$F=Z_Ef\in W^{\mabfp}_{Z,E,E_0,(\omega )}(\rr {2d})$,
$G=Z_Eg \in W^{\mabfp '}_{Z,E,E_0,(\omega )}(\rr {2d})$ and
$(F,G)_{Z,E}=(f,g)_{L^2(\rr d)}$.
\end{enumerate}
\end{thm}

\par

\begin{proof}
In order to prove (1) we first observe that if
$F\in W^{\mabfp}_{Z,E,E_0,(\omega )}(\rr {2d})$ and
$G\in W^{\mabfp '}_{Z,E,E_0,(1/\omega )}(\rr {2d})$, then
the integrand on the right-hand side of
\eqref{Eq:QuasiPerForm}$'$ belongs to $L^1(\Omega _E)$, in view of
Theorem \ref{Thm:ZakModulation}. This proves the extension
assertions for $(\cdo ,\cdo )_{Z,E}$ on
$W^{\mabfp}_{Z,E,E_0,(\omega )}(\rr {2d})
\times
W^{\mabfp '}_{Z,E,E_0,(1/\omega )}(\rr {2d})$. The duality assertion
will follow after we have proved (3), using the fact that the dual of
$M^{\mabfp}_{E\times E_0',(\omega _0)}(\rr d)$ is equal to
$M^{\mabfp '}_{E\times E_0',(1/\omega _0)}(\rr d)$
(see \cite[Theorem 11.3.6]{Gc2}).

\par

It remains to prove (2). Let $f\in M^{\mabfp}_{E\times E_0',(\omega _0)}(\rr d)$,
$g\in M^{\mabfp '}_{E\times E_0',(1/\omega _0)}(\rr d)$, $F=Z_Ef$
and $G=Z_Eg$. Then Theorem \ref{Thm:ZakModulation} shows that
$F\in W^{\mabfp}_{Z,E,E_0,(\omega )}(\rr {2d})$
and
$G\in W^{\mabfp '}_{Z,E,E_0,(1/\omega )}(\rr {2d})$.
By straight-forward computations it follows that
\eqref{Eq:LastFormEvaluation} holds for our choices of $f$, $g$, $F$ and
$G$. This gives (2).
\end{proof}

\par

\begin{rem}
Let $p\in [1,\infty ]$ and $E$ be an ordered basis of $\rr d$.
Then recall that we may identify $L^p_{Z,E}(\rr {2d})$ with
$L^p(Q_{E\times E'})$. Hence, by \eqref{Eq:QuasiPerScalarProd}
and standard duality properties for Lebesgue spaces show that
the map $(F,G)\mapsto (F,G)_{Z,E}$ from $C^\infty _{Z,E}(\rr {2d})
\times C^\infty _{Z,E}(\rr {2d})$ to $\mathbf C$ is uniquely extendable
to a continuous mapping from $L^p_{Z,E}(\rr {2d})\times L^{p'}_{Z,E}(\rr {2d})$
to $\mathbf C$. If in addition $p<\infty$, then the dual of
$L^p_{Z,E}(\rr {2d})$ can be identified with $L^{p'}_{Z,E}(\rr {2d})$
through the form $(\cdo ,\cdo )_{Z,E}$;
\end{rem}

\par

\begin{rem}
Let $E_0$ be an ordered basis of $\rr d$, $E=E_0\times E_0'$,
$\omega _0\in \mascP _E(\rr d)$, $\omega (x,\xi )=\omega _0(\xi )$,
$x,\xi \in \rr d$ and $\mabfq \in [1,\infty )^d$.
For periodic functions and distributions,
Theorem \ref{Thm:DualQuasiPer2} (2) together with
Proposition \ref{Prop:PerMod} correspond to 
\cite[Theorem 3.2]{ToNa}, which among others asserts that
the dual of $\maclE _{E}(\omega _0,\ell _{E_0'}^{\mabfq}(\Lambda _{E_0'}))$
is equal to $\maclE _{E}(1/\omega _0,\ell _{E_0'}^{\mabfq '}(\Lambda _{E_0'}))$
through a unique extension of the $L^2(Q_{E_0})$ form on
$\maclE _E^0(\rr d)\times \maclE _E^0(\rr d)$.

\par

Here we observe the misprint in
(2) in \cite[Theorem 3.2]{ToNa}, where it stays
$\maclE ^E(1/\omega ,\ell ^{\mabfq}_{\kappa (E')})$
instead of
$\maclE ^E(1/\omega ,\ell ^{\mabfq '}_{\kappa (E')})$
\end{rem}

\par

\section{Transitions of operators under the Zak
transform}\label{sec4}

\par

In this section we show how linear operators are transformed
by the Zak transform into corresponding operators
acting on quasi-periodic functions or distributions. We also present
a condition on linear operators which is both necessary and sufficient 
in order for these operators should map quasi-periodic elements
into quasi-periodic elements.

\par

Our results are described in the following two theorems, which explain
how linear operators acting on functions and distributions on $\rr d$ are
transfered by the Zak transform.
Especially the operator representation
\begin{multline}\label{Eq:ReprUDef}
(U_{y,\eta }F)(x,\xi )= e^{-i\scal y{\xi +\eta}}F(x+y,\xi +\eta )
\\[1ex]
\text{when}\quad 
F\in (\maclS _{s,\sigma}^{\sigma ,s} )'(\rr {2d}),\ y,\eta \in \rr d
\end{multline}
is important for characterizing such operators.

\par

\begin{thm}\label{Thm:ZakTransOpChar}
Let $s,\sigma >0$, $U_{y,\eta}$ be as in \eqref{Eq:ReprUDef} and
$T$ be a linear operator from $\maclS _s^\sigma (\rr d)$ to
$(\maclS _s^\sigma )'(\rr d)$ with kernel $K\in (\maclS _s^\sigma )'(\rr {2d})$. Then
there is a unique linear and continuous operator $T_Z$ from $\maclE ^{\sigma ,s}_{Z,E}(\rr {2d})$
to $(\maclE ^{\sigma ,s}_{Z,E})'(\rr {2d})$
such that $Z_E\circ T = T_Z\circ Z_E$, for every ordered basis $E$ of $\rr d$.
The kernel of $T_Z$ is given by
\begin{align}
K_Z(x,\xi ,y,\eta ) &= (2\pi )^{-\frac d2}\mascF (K(x-\cdo ,y-\cdo ))(\eta -\xi )
\in (\maclS _{\boldsymbol s}^{\boldsymbol \sigma} )'(\rr {4d}),
\label{Eq:ZakTransOpChar1}
\intertext{where $\boldsymbol s = (s,\sigma ,s,\sigma )$ and
$\boldsymbol \sigma = (\sigma ,s,\sigma ,s)$, and $T_Z$ fulfills}
T_Z\circ U_{y,\eta} &= U_{y,\eta}\circ T_Z, \qquad y,\eta \in \rr d.
\label{Eq:ZakTransOpChar2}
\end{align}
The same holds true with $\maclE _{Z,E}^{\sigma ,s;0}$ and
$(\maclE _{Z,E}^{\sigma ,s;0})'$, or with $C^\infty _{Z,E}$ and
$\mascS _{Z,E}'$ in place of $\maclE _{Z,E}^{\sigma ,s}$ and
$(\maclE _{Z,E}^{\sigma ,s})'$, respectively at each occurrence.
\end{thm}

\par

The converse of Theorem \ref{Thm:ZakTransOpChar} is the following

\par

\begin{thm}\label{Thm:ZakTransOpChar2}
Let $s,\sigma >0$, $T_Z$ be a linear and continuous map from
$\maclS _{s,\sigma}^{\sigma ,s} (\rr {2d})$ to
$(\maclS _{s,\sigma}^{\sigma ,s} )'(\rr {2d})$ with kernel $K_Z$
and such that
\eqref{Eq:ZakTransOpChar2} holds. Then the following is true:
\begin{enumerate}
\item there is a unique
$K\in (\maclS _s^\sigma )'(\rr {2d})$ such that
\eqref{Eq:ZakTransOpChar1} holds;

\vrum

\item if $E$ is an ordered basis of $\rr d$,
then $T_Z$ is uniquely extendable to a linear and
continuous operator
from $\maclE ^{\sigma ,s}_{Z,E}(\rr {2d})$
to $(\maclE ^{\sigma ,s}_{Z,E})'(\rr {2d})$;

\vrum

\item If $T$ is the linear operator with kernel $K$ in {\rm{(1)}} and
$E$ is an ordered basis of $\rr d$, then $Z_E\circ T = T_Z\circ Z_E$.
\end{enumerate}
The same holds true with $\Sigma _{s,\sigma}^{\sigma ,s}$,
$\maclE _{Z,E}^{\sigma ,s;0}$ and
$(\maclE _{Z,E}^{\sigma ,s;0})'$, or with $\mascS$, $C^\infty _{Z,E}$ and
$\mascS _{Z,E}'$ in place of $\maclS _{s,\sigma}^{\sigma ,s}$,
$\maclE _{Z,E}^{\sigma ,s}$ and
$(\maclE _{Z,E}^{\sigma ,s})'$, respectively at each occurrence.
\end{thm}

\par

\begin{proof}[Proof of Theorem \ref{Thm:ZakTransOpChar}]
We only prove the result when the involved spaces are given by
$\maclE _{Z,E}^{\sigma ,s}$ or $(\maclE _{Z,E}^{\sigma ,s})'$. The other cases
follow by similar arguments and are left for the reader.

\par

First suppose that $K_Z$ is given by
\eqref{Eq:ZakTransOpChar1}. Since pull-back results of the type
\cite[Theorem 6.1]{Ho1} for usual distribution, hold true
for Gelfand-Shilov distributions concerning linear pull-backs,
it follows that $K_1(x,y,z)\equiv K(x-z,y-z)$
belongs to $(\maclS _s^\sigma )'(\rr {3d})$ when $K\in (\maclS _s^\sigma )'(\rr {2d})$.
By Fourier transformation, it follows that $K_2(x,y,\xi ) \equiv \mascF (K_1(x,y,\cdo ))(\xi )$
$(\maclS _{s,s,\sigma}^{\sigma ,\sigma ,s})'(\rr {3d})$. By similar pull-back results it now
follows that
$$
K_3(x,y,\xi ,\eta ) \equiv K_2(x,y,\eta -\xi )\in
(\maclS _{s,s,\sigma ,\sigma}^{\sigma ,\sigma ,s,s})'(\rr {4d}).
$$
Since $K_Z(x,\xi ,y,\eta )=(2\pi )^{-\frac d2}K_3(x,y,\xi ,\eta )$, we get
$$
K_Z\in (\maclS _{s,\sigma ,s,\sigma}^{\sigma ,s,\sigma ,s})'(\rr {4d}),
$$
and the last property in \eqref{Eq:ZakTransOpChar1} follows.

\par

Let $E$ be an ordered basis of $\rr d$. We only prove $Z_E\circ T = T_Z\circ Z_E$
when $K\in \mascS (\rr {2d})$. The general result follows by similar arguments and
is left for the reader. We have
$$
(2\pi )^{\frac d2}\mascF ^{-1}(K_Z(x,\xi ,y,\cdo ))(z) = K(x-z,y-z)e^{i\scal z\xi}.
$$
This gives
\begin{multline*}
T_Z(Z_Ef)(x,\xi )
=
\sum _{j\in \Lambda _E}\iint _{\rr {2d}}K_Z(x,\xi ,y,\eta )
f(y-j)e^{i\scal j\eta}\, dyd\eta
\\[1ex]
=
\sum _{j\in \Lambda _E}(2\pi )^{\frac d2}\int _{\rr {d}}
\mascF ^{-1}(K_Z(x,\xi ,y,\cdo ))(j)f(y-j)\, dy
\\[1ex]
=
\sum _{j\in \Lambda _E}\int _{\rr {d}}
K(x-j,y-j)e^{i\scal j\xi}f(y-j)\, dy
= 
\sum _{j\in \Lambda _E}\int _{\rr {d}}
K(x-j,y)e^{i\scal j\xi}f(y)\, dy
\\[1ex]
=
\sum _{j\in \Lambda _E}(Tf)(x-j)e^{i\scal j\xi}
=Z_E(Tf)(x,\xi ).
\end{multline*}
This shows that $Z_E\circ T = T_Z\circ Z_E$.

\par

The continuity assertions of $T_Z$ now follows from the latter identity and
Theorem \ref{Thm:ZakDist}.
\end{proof}

\par

We need the following lemma for the proof of Theorem
\ref{Thm:ZakTransOpChar2}.

\par

\begin{lemma}\label{Lemma:ZakTransOpChar2}
Let $s,\sigma >0$ and $K\in (\maclS _s^\sigma )'(\rr {2d})$. Then
the following conditions are equivalent:
\begin{enumerate}
\item $K(\cdo +(z,z))=K$ for every $z\in \rr d$;

\vrum

\item there is a unique $K_0\in (\maclS _s^\sigma )'(\rr d)$ such that
$K(x,y) = K_0(x-y)$.
\end{enumerate}
The same hold true with $(\Sigma _s^\sigma )'$, $\mascS '$ or
$\mascD '$ in place of $(\maclS _s^\sigma )'$ at each occurrence.
\end{lemma}

\par

Lemma \ref{Lemma:ZakTransOpChar2} is at least implicitly
available in the literature, e.{\,}g. in \cite{Ho1}. In order to be
self-contained we give a proof in Appendix \ref{App:A}.

\par

\begin{proof}[Proof of Theorem \ref{Thm:ZakTransOpChar2}]
Again we only prove the result when the involved spaces are given by
$\maclE _{Z,E}^{\sigma ,s}$ or $(\maclE _{Z,E}^{\sigma ,s})'$. The other cases
follow by similar arguments and are left for the reader.

\par

The condition \eqref{Eq:ZakTransOpChar2} implies that
$$
K_Z(x,\xi ,y,\eta ) = e^{-i\scal z{\xi -\eta}}K_Z(x+z,\xi +\zeta ,y+z,\eta +\zeta ),
$$
for every $x,y,z,\xi ,\eta ,\zeta \in \rr d$. By Lemma \ref{Lemma:ZakTransOpChar2}
it follows that
$$
K_Z(x,\xi ,y,\eta ) = e^{i\scal y{\xi -\eta}} K_0(x-y,\xi -\eta )
$$
for some $K_0\in (\maclS _{s,\sigma}^{\sigma ,s})'(\rr {2d})$. It now follows that
$$
K(x,y) = (2\pi )^{\frac d2}(\mascF _2^{-1}K_0)(x-y,y)
$$
fullfils all required properties.
\end{proof}


\appendix

\section{}\label{App:A}

\par

In this appendix we present proofs of Proposition \ref{Prop:ZakOnLebesgue}
and Lemma.

\par

\begin{proof}[Proof of Proposition \ref{Prop:ZakOnLebesgue}]
First suppose that $p\in (0,1]$. Then
\begin{multline*}
\nm F{L^p_{Z,E}}^p = |Q_{E'}|^{-1}\iint _{Q _{E\times E'}}
\left | 
\sum _{j\in \Lambda _E}f(x-j)e^{i\scal j\xi}
\right | ^p\, dxd\xi
\\[1ex]
\le
|Q_{E'}|^{-1}\iint _{Q _{E\times E'}}
\left ( 
\sum _{j\in \Lambda _E}
|f(x-j)e^{i\scal j\xi}| ^p\right )
\, dxd\xi
\\[1ex]
=
\int _{Q_E}
\left ( 
\sum _{j\in \Lambda _E}
|f(x-j)| ^p\right )
\, dx = \nm f{L^p}^p,
\end{multline*}
and (1) follows for $p\in (0,1]$.

\par

Since the result is true for $p=2$ in view of Proposition
\ref{Prop:ZakTransfBasicMaps} below and Proposition
\ref{Prop:ZakTransfBasicMaps}, the result now follows
in the case $p\in [1,2]$ by interpolating the case $p=1$
above with the case $p=2$. This gives (1),  and (2) now
follows from (1) and duality.

\par

Finally, by the assumptions we have that $\sum _{j\in \Lambda _E}v(x+j)^{-1}$
is bounded. Hence, by the inversion formula \eqref{Eq:ZakTransfInv}
we obtain
\begin{multline*}
\nm {f\cdot v^{-1}}{L^1}
\asymp
\int _{\rr d} \left | \int _{Q_{E'}}F(x,\xi )\, d\xi \right |v(x)^{-1}\, dx
\\[1ex]
\le
\sum _{j\in \Lambda _E}
\int _{Q_E}  \left (\int _{Q_{E'}}|F(x+j,\xi )|\, d\xi \right )
v(x+j)^{-1}\, dx
\\[1ex]
=
\int _{Q_E}  \left (\int _{Q_{E'}}|F(x,\xi )|\, d\xi \right )
\left (\sum _{j\in \Lambda _E} v(x+j)^{-1}\right )\, dx
\asymp
\nm F{L^1_{Z,E}(\rr {2d})},
\end{multline*}
and (3) follows in the case $p=1$ by combining the previous estimate
with (1). Since (3) is true for $p=2$ in view of Proposition
\ref{Prop:ZakTransfBasicMaps}, it follows that it is true also for
$p\in [1,2]$ by interpolating between the cases $p=1$ and $p=2$.
For $p\in [2,\infty]$, (3) now follows from the case $p\in [2,\infty ]$
and duality.
\end{proof}

\par

\begin{proof}[Proof of Lemma \ref{Lemma:ZakTransOpChar2}]
We only prove the result when the involved spaces are of the forms
$(\maclS _s^\sigma )'(\rr d)$ and $(\maclS _s^\sigma )'(\rr {2d})$. The
other cases follow by similar arguments and are left for the reader.

\par

It is evident that (2) implies (1). Suppose that (1) is true. Then
$(x,y)\mapsto K(2^{-1}(x+y),2^{-1}(x-y))$ is an element in
$(\maclS _s^\sigma )'(\rr {2d})$ which is constant with respect to
the $x$ variable. Hence,
$$
K(2^{-1}(x+y),2^{-1}(x-y)) =(1\otimes K_0)(x,y)
$$
for some $K_0\in (\maclS _s^\sigma )'(\rr d)$. By taking
$2^{-1}(x+y)$ and $2^{-1}(x-y)$ as new variables, we obtain (2).
\end{proof}

\par

\end{document}